\documentclass[reqno,twoside,11pt]{amsart}

\usepackage{amsthm, amssymb, amsmath, amsfonts,eufrak}
\usepackage{verbatim}
\IfFileExists{epsf.def}{\input epsf.def}{\usepackage{epsf}}
\usepackage{mathptmx}
\usepackage{mathrsfs}

\DeclareFontFamily{OT1}{cmss}{} \DeclareFontShape{OT1}{cmss}{m}{n} {<5> <6> <7> <8> <9> <10> <11> <12> <13> <14.4> cmss10}{}
\DeclareMathAlphabet{\cmss}{OT1}{cmss}{m}{n}

\newtheoremstyle{thm}{1.8ex}{1.8ex}{\itshape\rmfamily}{} {\bfseries\rmfamily}{}{2ex}{}

\newtheoremstyle{def}{1.8ex}{1.8ex}{\slshape\rmfamily}{} {\bfseries\rmfamily}{}{2ex}{}

\newtheoremstyle{rem}{1.8ex}{1.8ex}{\rmfamily}{} {\bfseries\rmfamily}{}{2ex}{}

\newenvironment{proofsect}[1] {\vspace{0.2cm}\noindent{\rmfamily\itshape#1.}}{\qed\vspace{0.15cm}}
\theoremstyle{thm}
\newtheorem{theorem}{Theorem}[section]
\newtheorem{lemma}[theorem]{Lemma}
\newtheorem{proposition}[theorem]{Proposition}

\newtheorem*{Main Theorem}{Main Theorem.}
\newtheorem{corollary}[theorem]{Corollary}
\newtheorem*{special theorem}{Lindeberg-Feller Theorem for Martingales}

\newtheorem{conjecture}[theorem]{Conjecture}

\theoremstyle{def}
\newtheorem{definition}[theorem]{Definition}

\theoremstyle{rem}
\newtheorem{remark}[theorem]{Remark}
\newtheorem{remarks}[theorem]{Remarks}

\numberwithin{equation}{section}

\renewcommand{\small}{\fontsize{9}{9}\selectfont}

\renewcommand{\section}{\secdef\sct\sect}
\newcommand{\sct}[2][default]{%
\refstepcounter{section}
\addcontentsline{toc}{section}{{\tocsection {}{\thesection}{\!\!\!\!#1\dotfill}}{}}
\vspace{0.7cm}
\centerline{\scshape\thesection.\ #1} \nopagebreak \vspace{0.2cm}}
\newcommand{\sect}[1]{%
\vspace{0.4cm} \centerline{\large\scshape\rmfamily #1}
\vspace{0.2cm}}

\renewcommand{\subsection}{\secdef\subsct\sbsect}
\newcommand{\subsct}[2][default]{\refstepcounter{subsection}
\addcontentsline{toc}{subsection}
{{\tocsection{\!\!}{\hspace{1.2em}\thesubsection}{\!\!\!\!#1\dotfill}}{}}
\nopagebreak\vspace{0.45\baselineskip} {\flushleft\bf
\thesubsection~\bf #1.~}
\\*[3mm]\noindent
\nopagebreak}
\newcommand{\sbsect}[1]{\vspace{0.1cm}\noindent
\textbf{#1.~}\vspace{0.1cm}}

\renewcommand{\subsubsection}{%
\secdef \subsubsect\sbsbsect}
\newcommand{\subsubsect}[2][default]{%
\refstepcounter{subsubsection} 
\addcontentsline{toc}{subsubsection}{{\tocsection{\!\!}
{\hspace{3.05em}\thesubsubsection}{\!\!\!\!#1\dotfill}}{}}
\nopagebreak
\vspace{0.15\baselineskip} \nopagebreak {\flushleft\rmfamily
\itshape\thesubsubsection
\ \rmfamily #1\/.}\ }
\newcommand{\sbsbsect}[1]{\vspace{0.1cm}\noindent
\rmfamily \itshape
\arabic{section}.\arabic{subsection}.\arabic{subsubsection} \
\sffamily #1\/.\ }

\def\myffrac#1#2 in #3{\raise 2.6pt\hbox{$#3 #1$}\mkern-1.5mu\raise 0.8pt\hbox{$#3/$}\mkern-1.1mu\lower 1.5pt\hbox{$#3 #2$}}
\newcommand{\ffrac}[2]{\mathchoice%
{\myffrac{#1}{#2} in \scriptstyle}
{\myffrac{#1}{#2} in \scriptstyle}
{\myffrac{#1}{#2} in \scriptscriptstyle}
{\myffrac{#1}{#2} in \scriptscriptstyle}
}

\newcommand{\D}{\mathbb D}
\newcommand{\G}{\mathbb G}
\newcommand{\N}{\mathbb N}
\newcommand{\R}{\mathbb R}
\newcommand{\B}{\mathbb B}
\renewcommand{\H}{\mathbb H}
\newcommand{\PP}{\mathbb P}
\newcommand{\Z}{\mathbb Z}
\newcommand{\Q}{\mathbb Q}
\newcommand{\E}{\mathbb E}

\renewcommand{\AA}{\mathcal{A}} 
\newcommand{\scrC}{\mathscr{C}} 
\newcommand{\EE}{\mathcal{E}}
\newcommand{\FF}{\mathcal{F}}
\newcommand{\scrF}{\mathscr{F}}
\newcommand{\scrG}  {\mathscr{G}} 
\newcommand{\LL} {\mathcal{L}}

\newcommand{\twoeqref}[2]{(\ref{#1}--\ref{#2})}
\newcommand{\diam}{\text{\rm diam}}

\newcommand{\cc}{{\text{\rm c}}}
\newcommand{\texte}{{\text{\rm e}}}
\newcommand{\1}{\text{\bf 1}}
\newcommand{\textd}{\text{\rm d}}
\newcommand{\hate}{\hat{\text{\rm e}}}
\newcommand{\dist}{\text{\rm dist}}

\title[Heat-kernel decay in random conductance models]
{Subdiffusive heat-kernel decay in four-dimensional\\i.i.d.\ random conductance models}
\author[M.~Biskup and O.~Boukhadra]
{M.~Biskup$^{1,2}$ \and\: O.~Boukhadra$^{3}$}

\begin{document} 
\thanks{\small\hglue-4.5mm
\copyright\,\textrm{2012} \textrm{M.~Biskup and O.~Boukhadra}.
Reproduction, by any means, of the entire article for non-commercial purposes is permitted without charge.}

\maketitle
\vspace{-5mm}
\centerline{\textit{$^1$Department of Mathematics, UCLA, Los Angeles, California, USA}}
\centerline{\textit{$^2$School of Economics, University of South Bohemia, \v Cesk\'e Bud\v ejovice, Czech Republic}}
\centerline{\textit{$^3$D\'epartement de Math\'ematiques, UMC, Constantine, Algeria}}

\vspace{-2mm}
\begin{abstract}
We study the diagonal heat-kernel decay for the four-dimensional nearest-neighbor random walk (on~$\Z^4$) among i.i.d.~random conductances that are positive, bounded from above but can have arbitrarily heavy tails at zero. It has been known that the quenched return probability $\cmss P_\omega^{2n}(0,0)$ after $2n$ steps is at most $C(\omega) n^{-2} \log n$, but the best lower bound till now has been $C(\omega) n^{-2}$. Here we will show that the $\log n$ term marks a real phenomenon by constructing an environment, for each sequence $\lambda_n\to\infty$, such that 
$$
\cmss P_\omega^{2n}(0,0)\ge C(\omega)\log(n)n^{-2}/\lambda_n,
$$
with $C(\omega)>0$~a.s., along a deterministic subsequence of $n$'s. Notably, this holds simultaneously with a (non-degenerate) quenched invariance principle. As for the $d\ge5$ cases studied earlier, the source of the anomalous decay is a trapping phenomenon although the contribution is in this case collected from a whole range of spatial scales.
\end{abstract}



\section{{Introduction and Results}}
\noindent
Recent years have witnessed remarkable progress in the understanding of a class of reversible random walks in random environments that go under the name Random Conductance Model. The setting of a typical instance of this problem is as follows:  Consider the $d$-dimensional hypercubic lattice $\Z^d$ and let $\B^d$ denote the set of unordered nearest-neighbor pairs.  For  a configuration $\omega=(\omega_b)_{b\in\B^d} \in(0,\infty)^{\B^d}$, define the Markov chain $X=(X_n)_{n\geq0}$ with state space $\Z^d$ and transition probability
\begin{equation}
\label{protra}
\cmss P_{\omega}(x,y):=
\begin{cases}
\frac{\omega_{xy}}{\pi_{\omega}(x)},\qquad& (x,y)\in \B^d,\\
0,\qquad& \text{otherwise,}
\end{cases}
\end{equation}
where 
\begin{equation}
\pi_\omega(x):=\sum_{y\colon (x,y)\in\B^d}\omega_{xy}.
\end{equation}
Sometimes even~$\omega_b=0$ is permitted; the state space is then just $\{x\colon\pi_\omega(x)>0\}$ or, when such exists, an infinite connected component thereof. Let $P_\omega^x$ denote the distribution of~$X$ subject to $P_\omega^x(X_0=x)=1$. The principal items of interest are various asymptotics of the law of~$X$ under~$P_\omega^x$ in the situation when~$\omega$ is a sample from a probability distribution~$\PP$.

Much of the early effort by probabilists concerned the validity of the (functional) Central Limit Theorem. In a sequence of papers (Kipnis and Varadhan~\cite{Kipnis-Varadhan}, De Masi, Ferarri, Goldstein and Wick~\cite{demas1,demas2}, Sidoravicius and Sznitman~\cite{Sidoravicius-Sznitman}, Berger and Biskup~\cite{BB}, Mathieu and Piatnitski~\cite{Mathieu-Piatnitski}, Mathieu~\cite{Mathieu-CLT}, Biskup and Prescott~\cite{BP}, Barlow and Deu\-schel~\cite{Barlow-Deuschel}, Andres, Barlow, Deuschel and Hambly~\cite{ABDH}), it has gradually been established that, as $n\to\infty$, the law of $t\mapsto X_{\lfloor nt\rfloor}/\sqrt{n}$ under~$P_\omega^0$ scales to a non-degenerate Brownian motion for almost every environment $\omega$, provided that certain conditions are met by the law of~$\omega$. For \emph{i.i.d.}\ laws~$\PP$ concentrated on $[0,\infty)^{\B^d}$, in~$d\ge2$ these conditions are
\begin{equation}
\label{E:1.3}
\E (\omega_b) <\infty\quad\text{and}\quad  \PP(\omega_b>0)>p_\cc(d),
\end{equation}
where~$\E$ denotes the expectation in~$\PP$ and $p_\cc(d)$ is the critical threshold for bond percolation on~$\Z^d$. In $d=1$ the second condition needs to be replaced by $\E (\omega^{-1}_b)<\infty$; independence is not required (e.g., Biskup and Prescott~\cite{BP}). The same conditions as in $d=1$ are sufficient to imply the quenched CLT in $d=2$ for general ergodic environments (Biskup~\cite{Biskup-review}).

While the proof of the functional CLT is remarkably soft for the law on path space that is averaged over the environment --- the so called \emph{annealed} or \emph{averaged} law --- the almost-sure or \emph{quenched} law generally requires  also the heat kernel upper bound,
\begin{equation}
\label{Intro:HC-upper}
\cmss P^n_\omega (x,y)\leq \frac{c_1}{n^{d/2}}\,\texte^{-c_2\vert x-y\vert^2/n}, \quad n\geq N(\omega,x,y).
\end{equation}
This is conceptually wrong as the CLT seems to require a \emph{local}-CLT type of estimate. 
Notwithstanding, for environments possessing some form of uniform ellipticity, these heat-kernel upper bounds can be obtained. Indeed, they are the results of the sequence of papers by Delmotte~\cite{Delmotte}, Benjamini and Mossel~\cite{Benjamini-Mossel}, Heicklen and Hoffman~\cite{Heicklen-Hoffman}, Mathieu and Remy~\cite{Mathieu-Remy} culminating in Barlow's work~\cite{Barlow} for the simple random walk on the supercritcal percolation cluster. (We regard this case as uniformly elliptic too although on a spatially inhomogeneous graph.) 
However, in the environments with heavy tails at zero, it was in fact discovered that \eqref{Intro:HC-upper} may fail (Fontes and Mathieu~\cite{Fontes-Mathieu}, Berger, Biskup, Hoffman and Kozma~\cite{BBHK}) and a coarse-graining procedure was required to overcome this difficulty and derive the quenched functional CLT (Mathieu~\cite{Mathieu-CLT}, Biskup and Prescott~\cite{BP}). We note that when the left condition in \eqref{E:1.3} fails, the scaling limit of~$X$ may be not be diffusive at all (Barlow and \v Cern\'y~\cite{Barlow-Cerny}, Barlow and Zheng~\cite{Barlow-Zheng}, \v Cern\'y~\cite{Cerny}).

The study~\cite{BBHK} presents two types of results. First, for i.i.d.\ environment laws bounded from above, it restricts the diagonal heat-kernel decay by the following estimates
\begin{equation}
\label{E:me}
\cmss P_\omega^n(0,0)\le C(\omega)\,
\begin{cases}
n^{-d/2},\qquad&d=2,3,
\\
n^{-2}\log n,\qquad&d=4,
\\
n^{-2},\qquad&d\ge5,
\end{cases}
\end{equation}
where $C(\omega)<\infty$ almost surely, with the additional observation,
\begin{equation}
\label{E:Noam}
n^2\cmss P^{2n}_\omega(0,0)\xrightarrow[n \to \infty]{}0\quad \PP\text{-a.s.}\qquad\text{in }d\ge5.
\end{equation}
Second, for any sequence $\lambda_n\uparrow\infty$, an i.i.d.\ environment law is constructed so that
\begin{equation}
\label{E:Gaddy}
\cmss P_\omega^{2n_k}(0,0)\geq \frac{C(\omega)}{\lambda_{n_k}n^2_k},\quad k\geq 1,
\end{equation}
along a deterministic sequence $n_k\rightarrow\infty$, where $C(\omega)>0$ almost surely. Since the Central Limit Theorem for $X_n$ holds, we also have
\begin{equation}
\label{E:1.8}
\cmss P_\omega^{2n}(0,0)\geq \frac{C(\omega)}{n^{d/2}},
\end{equation}
cf, e.g.,~\cite[Remark~2.2]{BP}.

Putting the bits and pieces together we conclude that the return probability $\cmss P_\omega^{2n}(0,0)$ \emph{always} decays diffusively in spatial dimensions $d=2,3$, while, in dimensions $d\geq 5$, it can decay as slow as $o(n^{-2})$. (In $d=1$, the decay can be arbitrarily slow.) Further progress has been made by Boukhadra~\cite{B1,B2} who showed that the transition from regular decay $n^{-d/2}$ to anomalous decay $n^{-2+o(1)}$ in $d\geq5$ actually occurs in the class of power-law tails --- with the exponent $\gamma=d/2$ in $\PP(0<\omega_b<s)\sim s^{\gamma}$ being presumably the critical for the anomaly to appear. In $d\ge5$ this meshes nicely with the annealed estimates obtained by Fontes and Mathieu~\cite{Fontes-Mathieu}.

\smallskip
The combined results of~\cite{BBHK,B1,B2} provide definitive answers in all  spatial dimensions except $d=4$, where  \eqref{E:me} and \eqref{E:Gaddy} differ by a logarithmic factor. Computations for time-dependent environments (cf.~Theorem~5.3 of \cite{BBHK}) suggested that \eqref{E:me} is presumably the one closer to the truth, but any feasible method of proof seemed to require control of \textit{off-diagonal} heat-kernel \emph{lower} bounds. This would seem in turn to demand --- in order to avoid circuitous reasoning --- running a complicated induction along scales. We are thus pleased to report on a conceptually straightforward, albeit still technically complicated, proof of the following theorem:

\begin{theorem}
\label{Th}
Assume $d=4$. For every sequence $\{\lambda_n\}$ with $\lambda_n\uparrow\infty$, there exists an i.i.d.\ environment law~$\PP$ with $\PP (0<\omega_b\leq1)=1$, a random variable $C(\omega)$ with $\PP(C>0)=1$ and a sequence  $n_k\rightarrow\infty$ such that for every $n\in\{n_k\}$,
\begin{equation}
\label{E:main}
\cmss P_\omega^{2n}(0,0)\geq C(\omega)\frac{\log n}{\lambda_{n}n^2}.
\end{equation}
\end{theorem}

The conclusion is that the possibility of anomalous heat-kernel decay in the random conductance model extends to all dimensions $d\geq 4$; yet in $d=4$ the correction to diffusive behavior is only logarithmic. We actually believe, although cannot prove, that similarly to $d\ge5$ the upper bound in \eqref{E:me} can be approached arbitrarily closely but will not be attained in any given environment. We formulate this as a conjecture:

\begin{conjecture}
Assume $d=4$. For every i.i.d. environment $\PP$ with $\PP (0<\omega_b\leq 1)=1$,
\begin{equation}
\label{conj}
\frac{n^2}{\log n}\cmss P_\omega^{2n}(0,0)\xrightarrow[n \to \infty]{}0, \quad \PP\text{\rm-a.s.}
\end{equation}
\end{conjecture}

We point out that the proof of the asymptotic \eqref{E:Noam} in $d\ge5$ does not seem to carry over to this case and so presumably a new idea is needed here. \textit{Update in revised version}: The conjecture has in the meantime been proved in an upcoming preprint by Biskup, Louidor, Rozinov and Vandenberg-Rodes~\cite{UCLA-team}. A novel input is the use of the Dominated Ergodic Theorem.

\smallskip
Here is the main idea underlying the proof of Theorem~\ref{Th}. In order to ensure that the chain returns to its starting point at a required time, we can either let it arrive there more or less by accident (the dominating strategy in $d=2,3$) or make it fall into (and hide inside) a specific trap nearby which makes a later return to the starting point considerably less difficult (the dominating strategy in $d\ge5$). However, in $d=4$ the difference between the two strategies is so subtle that we can no longer force which trap the chain falls into; it needs to find one by itself. This puts us in the class of ideas underlying the approach taken by Boukhadra~\cite{B1,B2} to control the $d\ge5$ anomaly in the class of power-law tails.

\begin{remarks}
(1)
Should one desire to have \eqref{E:main} without a random constant, this can be done by making the range of~$k$'s random. 

(2) We did not try to optimize the tails of the distribution of $\omega_b$ for which the anomalous behavior occurs in $d=4$ although we think it is unlikely to occur in the class of power-law tails. On the other hand, our method of proof --- being closer to the approach of~\cite{B1,B2} --- could conceivably be adapted to yield a proof of a sharp threshold in the exponents that would yield anomalous behavior in~$d\ge5$. 

(3)
As in~\cite{BBHK}, we made no attempt to derive off-diagonal estimates on the heat kernel. Thus, in all cases~$d\ge4$, a question remains how one can reconcile the subdiffusive diagonal heat-kernel decay with the standard heat-kernel decay that should resume validity (at least on average) at the diffusive space-time scale (by the CLT).

(4) The proof of the upper bounds \eqref{E:me} in~\cite{BBHK} applies even to non-i.i.d.\ environments that possess a ``strong'' component almost surely. (This can be guaranteed by requiring that the law is dominated from below by an i.i.d.~law in which edges with positive conductances percolate.) On the other hand, it is not hard to construct correlated environments that would make the heat kernel decay arbitrarily slowly (in any~$d\ge1$).
A question remains whether there are some robust (e.g., moment) conditions that would imply a ``standard'' diffusive decay regardless of correlations.

(5) An interesting question is what strategy dominates the event $\{X_{2n}=0=X_0\}$ when anomalous decay occurs. In particular, does the walk visit at most one trap during its run or a sequence of traps (with various ``strengths'' of trapping), etc? \textit{Update in revised version}: This question has also been resolved in the aforementioned preprint~\cite{UCLA-team}; the path does spend a majority of its time in a few very localized places.

(6) It would be interesting to see whether and how the \emph{sub}diffusivity of the random walk among random conductances manifests itself in the behavior of its \emph{loop-erasure}. Here we note that, for instance, in $d=2$ the scaling limit of the loop-erasure of the simple random walk on the supercritical percolation cluster coincides with that for the ordinary simple random walk --- namely SLE$_2$ (Yadin and Yehudayoff~\cite{YY}). The question is thus whether one can extend this remarkable result to other dimensions (obviously, with a different scaling limit) and other conductance laws.
\end{remarks}

The plan of the remainder of this paper is as follows. In the next section we collect the ideas entering the proof and structure the main steps into proper lemmas and propositions. The proof of the main theorem then can be given subject to a Key Lemma (Lemma~\ref{key-lemma}) that controls the number of traps the chain typically sees along its path. The Key Lemma is then subsequently reduced to moment bounds on the corresponding number for a coarse-grained chain; this is done in Section~\ref{sec3}. In Section~\ref{sec4} we then invoke certain technical facts about the heat kernel for the coarse-grained chain, and also the trap density, to justify these moment bounds. These technical facts are then proved in Section~\ref{sec5} (trap density bounds) and Sections~\ref{sec6}-\ref{sec7} (heat-kernel estimates).

\section{{Key steps of the proof}}
\noindent
A distinguished feature of the four-dimensional problem, and the reason why the heat-kernel anomaly is manifested only by logarithmic corrections, is that the leading contribution to return probability may come from a whole range of spatial scales. Anticipating some form of scale invariance, we partition~$\Z^d$ into a sequence of (disjoint) annuli
\begin{equation}
B_k:=\bigl\{ x\in\Z^d\colon 2^{k-1}-1< \vert x\vert_\infty< 2^k\bigr\},\quad k\ge0.
\end{equation}
Let~$|B_k|$ denote the cardinality of~$B_k$.
An opening step of the proof is the following version of a standard (deterministic) Cauchy-Schwarz estimate. 

\begin{lemma}
\label{L:Cauchy-Schwarz}
Suppose that $0<\omega_b\le1$ for all $b$. Then
\begin{equation}
\label{l21}
\cmss P^{2n}_\omega(0,0)\geq \frac{\pi_\omega(0)}{2d} \sum_{k\ge0} \frac{P^0_\omega(X_n\in B_k)^2}{\vert B_k\vert}.
\end{equation}
\end{lemma}

\begin{proof}
By the Markov property and reversibility
\begin{equation}
\begin{alignedat}{2}
\cmss P^{2n}_\omega(0,0)
&=&
\sum_{x\in\Z^d} \cmss P^n_\omega(0,x) \cmss P^n_\omega(x,0)\\
&=&
\sum_{x\in\Z^d} \cmss P^n_\omega(0,x)^2\,\frac{\pi_\omega(0)}{\pi_\omega(x)}. 
\end{alignedat}
\end{equation}
Bounding $\pi_\omega(x)\leq 2d$ and using that $\{B_k\}_{k\geq 1}$
form a partition of $\Z^d$, we get
\begin{equation}
\cmss P^{2n}_\omega(0,0)\geq \frac{\pi_\omega(0)}{2d} \sum_{k\geq0 }\sum_{x\in B_k} \cmss P^n_\omega(0,x)^2.
\end{equation}
By Cauchy-Schwarz, the sum over~$x$ exceeds $\vert B_k\vert^{-1} P^0_\omega(X_n\in B_k)^2$.
\end{proof}

In order to motivate our next step, we recall a classic argument (cf, e.g.,~\cite[Remark~2.2]{BP}) that shows how the CLT implies \eqref{E:1.8}. Indeed, for $k:=\lfloor\frac{1}{2}\log_2(n)\rfloor$ and $n\gg1$ we have that $\text{\rm diam}(B_k)\sim\sqrt n$ and so the \emph{quenched} CLT gives
\begin{equation}
P^0_\omega(X_n\in B_k)\geq C_1(\omega)>0, \qquad \PP\text{\rm-a.s.}
\end{equation}
Retaining only the corresponding term in the sum, from $\vert B_k\vert\le C'n^{d/2}$ we get
\begin{equation}
\cmss P^{2n}_\omega(0,0)\geq \frac{C_2(\omega)}{n^{d/2}}\underset{d=4}= \frac{C_2(\omega)}{n^2},
\end{equation}
where $C_2>0$ a.s. This is the standard diffusive decay. The key idea underlying our work is that, in $d=4$ there are environments for which order-$\log n$ other $k$'s contribute a comparable amount to \eqref{l21} --- thus producing a $\log n$ multiplicative term. 

In order to state the requisite lower bound on $P^0_\omega(X_n\in B_k)$ in dimension-independent form, consider the abbreviation
\begin{equation}
\label{tk}
t_k:=2^{2k},\qquad k\ge0,
\end{equation}
and note that this is the diffusive time scale associated with the spatial scale of~$B_k$. 

\begin{proposition}
\label{pro22}
Let $d\ge4$ and consider an i.i.d.\ law~$\PP$ satisfying $\PP(0<\omega_b\le1)=1$ for which there exists a sequence $n_\ell\to\infty$ such that the quantity
\begin{equation}
\label{rho-n}
\rho_n:=\PP(\omega_b\ge\ffrac12)\,\PP \left(\ffrac{1}{n}\leq \omega_b\leq \ffrac{2}{n} \right)^{4d-2}
\end{equation}
obeys
\begin{equation}
\label{rho-n-decay}
\rho_{n_\ell}\log n_\ell\,\xrightarrow[\ell\to\infty]{}\infty.
\end{equation}
There are random variables $C_1=C_1(\omega)$ and $N_1=N_1(\omega)$, with $C_1(\omega)>0$ and $N_1(\omega)<\infty$ $\PP$-a.s., such that for all $n\in\{n_\ell\}_{\ell\ge1}$ with $n\ge N_1(\omega)$ and all $k\geq 1$ satisfying
\begin{equation}
\label{E:2.24}
\texte^{(\log\log n)^2}\le t_k\le \frac n{\log n}
\end{equation}
we have
\begin{equation}
\label{pro22eq}
P^0_\omega(X_n\in B_k)\geq C_1(\omega) \rho_n \frac{t_k}{n}.
\end{equation}
\end{proposition}

\begin{remark}
Note that the fact that~$\rho_n$ is summable on~$n$ --- which is seen, e.g., from the bound $\rho_n\le\PP(\ffrac{1}{n}\leq \omega_b\leq \ffrac{2}{n})$ --- forces us to work with subsequences in \eqref{rho-n-decay}. On the other hand, the requirement of subpolynomial decay of~$\rho_n$ is convenient, albeit perhaps unnecessary, for our proofs. (Specifically, this assumption is used in Lemmas~\ref{L:density-lower}  and \ref{lemma-sum-upper} and also in the proof of \eqref{E:3.14} from Lemmas~\ref{L:G-upper}--\ref{lemma-sum-upper}.) All of these steps will need to be reevaluated when studying the question for what tails does the anomaly start to occur in~$d=4$. (In $d\ge5$, this question has been addressed by Boukhadra~\cite{B1,B2}.)
\end{remark}

Proposition~\ref{pro22} permits us to finish the proof of our main result. It is important to note that in (and only in) $d=4$ we have that $t_k^2\sim|B_k|$ which puts all terms in the sum in \eqref{l21} on the same order of magnitude. 

\begin{proofsect}{Proof of Theorem~\ref{Th} from Proposition~\ref{pro22}}
Abbreviate $\theta:=1/[2(4d-2)]$ and suppose $d=4$. We may assume without loss of generality that $\lambda_n$ tends to infinity so slowly that
\begin{equation}
\label{E:2.10}
\lambda^{-\ffrac12}_n\log n \xrightarrow[n \to\infty]{}\infty.
\end{equation}
Let $\{n_\ell\}_{\ell\ge1}$ be an increasing sequence of integers with $n_\ell>1$ and such that
\begin{equation}
\sum_{\ell\geq 1}\lambda^{-\theta}_{n_\ell}\le\frac12.
\end{equation}
We then define the environment law~$\PP$ to be an i.i.d.\ measure whose one-dimensional marginals are concentrated on $\{1\}\cup\{n^{-1}_\ell\}_{\ell\ge1}$ with probabilities
\begin{equation}
\PP(\omega_b=n^{-1}_\ell):=\lambda^{-\theta}_{n_\ell}, \qquad \ell\geq1,
\end{equation}
and
\begin{equation}
\PP(\omega_b=1):=1-\sum_{\ell\geq1}\lambda^{-\theta}_{n_\ell}.
\end{equation}
Note that, for this environment,
\begin{equation}
\label{rholb}
\rho_{n_\ell}^2\geq\PP(\omega_b=1)^2\PP(\omega_b=n_\ell^{-1})^{2(4d-2)}
\ge\frac14\lambda^{-1}_{n_\ell}, \qquad \ell\geq1,
\end{equation}
and so \eqref{rho-n-decay} is implied for the subsequence $\{n_\ell\}$ by \eqref{E:2.10} and \eqref{rholb}.

Now pick $n\in \{n_\ell\}_{\ell\geq1}$ with $n\ge N_1(\omega)$, where $N_1$ as in Proposition~\ref{pro22}, and introduce the shorthand $\mathcal Z(n):=\{k\in\N\colon \texte^{(\log\log n)^2}\le t_k\le n/\log(n)\}$. Lemma~\ref{L:Cauchy-Schwarz} and Proposition~\ref{pro22} imply
\begin{eqnarray}
\label{smile}
\cmss P_\omega^{2n}(0,0)
&\geq&
 \frac{\pi_\omega(0)}{8} \sum_{k\in\mathcal Z(n)}
\dfrac{P^0_\omega(X_n\in B_k)^2}{\vert B_k\vert}\nonumber\\
&\geq&
\frac{\pi_\omega(0)}{8}C_1(\omega)^2\Bigl(\frac{\rho_n}{n}\Bigr)^2\,2^{-4}\bigl|\mathcal Z(n)\bigr|,
\end{eqnarray}
where we used that $t_k^2/|B_k|\ge 2^{-4}$.
For $n\gg1$, we have $|\mathcal Z(n)|\ge\frac12\log_4 n\ge\frac14\log n$.
Hence, for $n\in \{n_ \ell\}_{\ell\geq 1}$ sufficiently large,
 \begin{equation}
\cmss P^{2n}_\omega(0,0)\geq\dfrac{\pi_\omega(0)C_1(\omega)^2}{512}\,\dfrac{\log n}{\lambda_n n^2}.
\end{equation}
This and the fact that $\cmss P_\omega^{2n}(0,0)>0$ for all $n\ge1$ imply the claim.
\end{proofsect}

It remains to construct the proof of Proposition~\ref{pro22}. As in the examples showing anomalous decay in $d\geq 5$, a mechanism that could make $P_\omega^0(X_n\in B_k)$ large even when $n\gg t_k$ (which is outside the central-limiting scaling) is to let the walk fall into a \textit{trap}. In analogy with \cite{BBHK,B1,B2}, we adopt the following (somewhat arbitrary) definition:

\begin{definition}
A trap at scale $n$ is an edge $b=(y,z)$ such that $\omega_b\ge\ffrac12$ and such that 
for any edge $b'\neq b$ incident with either $y$ or $z$,
\begin{equation}
\frac{1}{n}\leq \omega_{b'} \leq \dfrac{2}{n}.
\end{equation}
\end{definition}

Let $\AA_n(x)$ be the event on the space of environments that $x$ is a vertex \emph{neighboring} a trap edge at scale $n$. Let us abbreviate
\begin{equation}
B_k^\circ:=\bigl\{x\in\Z^d\colon 2^{k-1}+2< \vert x\vert_\infty< 2^k-3\bigr\}
\end{equation}
and note that $B_k^\circ\subset B_k$ and, in fact, $\text{dist}(B_k^\circ,B_k^\cc)\ge3$. In particular, if $\AA_n(x)$ occurs for $x\in B_k^\circ$, then the corresponding trap(s) and the edges incident therewith all lie in~$B_k$. The effect of trapping is captured by the next estimate:

\begin{lemma}
Let~$d\ge1$.
There is an absolute constant $c_1=c_1(d)>0$ such that for all $n,k\ge1$,
\begin{equation}
\label{prob-expect}
P_\omega^0(X_n\in B_k)\ge\frac{c_1}n\,
E^0_\omega \left(\sum^{n/2-1}_{\ell=0}\1_{\{X_\ell\in B_k^\circ\}}\1_{\AA_n(X_\ell)}\right).
\end{equation}
\end{lemma}

\begin{proofsect}{Proof}
For $x\in\Z^d$ such that $\AA_n(x)$ occurs, let $(y,z)$ be the trap edge that makes $\AA_n(x)$ occur. (In the presence of more such edges next to~$x$, we pick the one that is smallest in a fixed complete order on~$\B^d$.) We assume that this edge is labeled so that $x$ and $y$ are neighbors in~$\Z^d$. For $\ell\geq0$, we use $D_n(x,\ell)$ to denote the event
\begin{equation}
\label{Dn-def}
D_n(x,\ell):=\{X_\ell=x\}\cap \{X_{\ell+1}=y\}\cap \bigcap_{m=\ell+1}^n\bigl\{X_m\in\{y,z\}\bigr\}.
\end{equation}
First we note that
\begin{equation}
\label{wit}
\{X_n\in B_k\}\supset\bigcup_{\begin{subarray}{c}
x\in B^\circ_k\\ \AA_n(x)\,\text{occurs}
\end{subarray}}
\bigcup^{n/2-1}_{\ell=0} D_n(x,\ell).
\end{equation}
Indeed, on $D_n(x,\ell)$ (with~$x$ and~$\ell$ in the unions above) the walk at time~$n$ is at one of the endpoints of the trap, which are both in $B_k$ by the restriction $x\in B_k^\circ$.

Next we claim that the unions in \eqref{wit} are disjoint, i.e., $D_n(x,\ell)\cap D_n(x',\ell')=\emptyset$ for any pairs of indices  $(x,\ell)\ne(x',\ell')$ contributing to \eqref{wit}. This is because on $D_n(x,\ell)$, the walk spends more than half of its time crossing a single (trap) edge --- namely, $(y,z)$ in 
\eqref{Dn-def}. This walk must have entered the trap from vertex~$x$ at time~$\ell$ and so if $(x',\ell')$ is distinct from $(x,\ell)$, it cannot belong to $D_n(x',\ell')$. We conclude
\begin{equation}
\label{witbis}
P^0_\omega(X_n\in B_k)\geq \sum_{x\in B^\circ_k}\1_{\AA_n(x)}
 \sum^{n/2-1}_{\ell=0} P^0_\omega\bigl(D_n(x,\ell)\bigr).
\end{equation} 
The Markov property and a simple calculation imply
\begin{equation}
\label{E:2.25}
P^0_\omega\bigl(D_n(x,\ell)\bigr)\geq  P^0_\omega(X_\ell=x) \frac{1}{2 dn}\left(1+\frac{4(2d-1)}{n}\right)^{\ell-n}.
\end{equation}
The last two terms are at most $c_1/n$ for $c_1:=(2d)^{-1}\texte^{-4(2d-1)}$. Once \eqref{E:2.25} is used for all terms in \eqref{witbis}, the sums combine into the desired expectation. 
\end{proofsect}

The proof of Proposition~\ref{pro22} is now reduced to the following Key Lemma:

\begin{lemma}[Key Lemma]
\label{key-lemma}
Let~$d\ge4$. For any i.i.d.\ law $\PP$ satisfying $\PP(0<\omega_b\le1)=1$ and \eqref{rho-n-decay} for a sequence $\{n_\ell\}_{\ell\ge1}$, there are $\PP$-a.s.\ finite and positive random variables $C_2:=C_2(\omega)$ and $N_1=N_1(\omega)$ such that for all $n\in\{n_\ell\}_{\ell\ge1}$ with $n\ge N_1$ and all~$k$ obeying \eqref{E:2.24} we have
\begin{equation}
\label{keyeq}
E^0_\omega \left(\sum^{n/2-1}_{\ell=0}\1_{\{X_\ell\in B_k^\circ\}}\1_{\AA_n(X_\ell)}\right)\geq C_2(\omega)\rho_n t_k.
\end{equation}
\end{lemma}

\begin{proofsect}{Proof of Proposition~\ref{pro22} from Key Lemma}
The conditions on $k$ and $n$ are identical, and combining \eqref{keyeq} with \eqref{prob-expect} we get \eqref{pro22eq} with $C_1(\omega):=c_1C_2(\omega)$.
\end{proofsect}

\section{Proof of Key Lemma}
\label{sec3}\noindent
Our proof of the Key Lemma will require introduction of some technical tools that we will first try to motivate by giving a heuristic argument why \eqref{keyeq} should hold true. 

Recall the notation~$t_k$ from \eqref{tk}. By reducing the sum in \eqref{keyeq} to $t_k\le\ell\le 2t_k$ --- which is allowed because $t_k\ll n$ by the assumptions \eqref{E:2.24} --- the expectation in \eqref{keyeq} pertains to paths of the random walk on temporal scale $t_k$ and spatial scale~$\sqrt{t_k}$. This is a diffusive scaling so one might expect that the law of~$X_\ell$ will be already close to the stationary distribution, and thus more or less uniformly distributed, over~$B_k^\circ$. The expectation of each term in the (reduced) sum should therefore be bounded below by a constant times $\PP(\AA_n(0))$. As
\begin{equation}
\PP(\AA_n(0))\ge\rho _n,
\end{equation}
and as there are order~$t_k$ terms in the (reduced) sum, this would yield \eqref{keyeq}.

A fundamental problem with this reasoning is that, due to the presence of very weak bonds, the law of $X_\ell$ in~$B^\circ_k$ for $t_k\leq\ell\leq2t_k$ will \emph{not} be close to the stationary distribution at the required level. After all, the sole purpose of this note is to demonstrate the failure of a local-CLT scaling! As in \cite{BP,BBHK,Mathieu-CLT,B1,B2}, we will circumvent this problem by observing the walk only on a \textit{strong component}; namely, the connected component of edges~$b$ with $\omega_b\geq \alpha$ for some small enough $\alpha$ to be chosen momentarily. This walk already has good mixing properties but, unfortunately, the reduction of the expectation in \eqref{keyeq} to this walk involves a time change that will now need to be controlled as well. And as this happens on the background of an expectation of a (large) random variable, we will have to control moments of this random variable as well.

We now begin to formulate the aforementioned technical aspects precisely. Following up on earlier work \cite{BP,BBHK,B1,B2}, we will introduce a cutoff $\alpha$ and examine the connectivity properties of the graph~$\G_{\omega}$ with vertices~$\Z^d$ and edges $\{b\colon\omega_b\ge\alpha\}$. The key facts we will need are as follows:

\begin{proposition}
\label{P:perc}
Assume~$d\ge2$. Then there is~$p_0=p_0(d)\in(0,1)$ such that whenever $\PP(\omega_b\ge\alpha)\ge p_0$, then the following holds $\PP$-a.s.:
\begin{enumerate}
\item[(1)]
The graph $\G_\omega$ contains a unique infinite connected component $\scrC_{\infty,\alpha}=\scrC_{\infty,\alpha}(\omega)$.
\item[(2)]
The complement $\Z^d\setminus\scrC_{\infty,\alpha}$ has only finite connected components.
\end{enumerate}
If $\scrF_x$ denotes the connected component of~$\Z^d\setminus\scrC_{\infty,\alpha}$ containing~$x$ (with $\scrF_x=\emptyset$ for~$x\in\scrC_{\infty,\alpha}$) and $\dist_\omega(x,y)$ is the shortest-path distance measured on~$\scrC_{\infty,\alpha}$ then also:
\begin{enumerate}
\item[(3)]
Almost surely on~$\{0\in\scrC_{\infty,\alpha}\}$, 
\begin{equation}
\label{E:distance-comp}
\limsup_{|x|\to\infty}\frac{\dist_\omega(0,x)}{|x|}<\infty.
\end{equation}
\item[(4)]
If~$\diam_\omega(\scrF_x)$ denotes the maximum of $\dist_\omega(y,z)$ over all pairs of $\Z^d$-neighbors $y,z\in\scrC_{\infty,\alpha}$ of~$\scrF_x$, then $\diam_\omega(\scrF_0)$ has all moments. (Naturally, $\diam_\omega(\emptyset)=0$.)
\end{enumerate}
Finally, let~$\scrG_x$ denote the union of~$\scrF_y$ for~$y$ running through neighbors of~$x$ in~$\Z^d$. Let~$\G_\omega'$ denote the graph obtained from~$\scrC_{\infty,\alpha}$ by adding an edge between any~$y,z\in\scrC_{\infty,\alpha}$ with $\scrG_y\cap\scrG_z\ne\emptyset$ and let $\textd'_\omega(x,y)$ denote the graph-theoretical distance measured on~$\G_\omega'$. Then:
\begin{enumerate}
\item[(5)]
For some~$\xi>0$,
\begin{equation}
\limsup_{|x|\to\infty}\frac1{|x|}\log\PP\Bigl(\,0,x\in\scrC_{\infty,\alpha}\,\,\&\,\,\textd'_\omega(0,x)\le\xi|x| \Bigr)<0.
\end{equation}
\end{enumerate}
Here and henceforth, $|x|$ denotes the Euclidean norm of~$x$.
\end{proposition}

\begin{proofsect}{Proof (Sketch)}
First note that all properties (2-5) are properly stochastically monotone in~$\alpha$. (An exception is the uniqueness of~$\scrC_{\infty,\alpha}$ which holds for all~$\alpha$ by an argument of, e.g., Burton and Keane~\cite{Burton-Keane}.) Our proof is best explained by running a coarse-graining argument: Consider bond percolation on~$\Z^d$ with parameter~$p$ and call a unit cube of~$2^d$ vertices in~$\Z^d$ \emph{occupied} if all of its edges are occupied. Two cubes are called \emph{adjacent} if they share a side. By the result of  Liggett, Stacey and Schonmann~\cite{LSS}, the fact that the cubes more than distance one apart are independent permits us dominate the process of occupied cubes from below by site percolation on~$\Z^d$ with a parameter~$\eta(p)$, where~$\eta(p)\uparrow1$ when~$p\uparrow1$. In particular, there is $p_0\in(0,1)$ such that for all~$p\ge p_0$, the occupied cubes percolate and the removal of the (a.s.\ unique) infinite component of occupied cubes results only in finite components whose  diameters have an exponential~tail. 

Properties (1,2) now follow immediately by standard facts about percolation on~$\Z^d$ while (4) is the consequence of the fact that $\diam_\omega(\scrF_x)$ will be bounded by the number of unit cubes adjacent to the finite component of the cube-process (necessarily) containing~$\scrF_x$. Property~(3) is a consequence of Theorem~1.1 of~\cite{Antal-Pisztora} while property~(5) is a restatement of Lemma~3.1 of~\cite{BP}.
\end{proofsect}

Now let us fix~$p_0$ as in Proposition~\ref{P:perc} and pick~$\alpha_0\in(0,1)$ by
\begin{equation}
\PP(\omega_b\ge\alpha_0)\ge p_0.
\end{equation}
We will keep~$\alpha_0$ fixed throughout the rest of the paper. Note that the properties (1-5) in Proposition~\ref{P:perc} apply to all cutoffs $\alpha\in(0,\alpha_0]$.

Consider now a path of the Markov chain~$X$. For any~$\omega$ with $0\in\scrC_{\infty,\alpha}$, we define a sequence $T_0:=0, T_1, T_2, \ldots$ via
\begin{equation}
\label{E:Ts}
T_{j+1}:=\inf\{\ell>T_0+\cdots+ T_j~: X_\ell\in \scrC_{\infty,\alpha}\}-(T_0+\cdots+T_j)
\end{equation}
and
\begin{equation}
\hat{X}_\ell:=X_{T_1+\cdots+T_\ell}, \qquad \ell\geq0.
\end{equation}
The sequence $(\hat{X}_\ell)_{\ell\geq1}$ records the successive visits of $(X_n)$ to the strong component $\scrC_{\infty,\alpha}$. Note that we have $T_j<\infty$ for all~$j\ge0$, $P^0_\omega$-a.s. In fact, there is a (deterministic) moment bound on the time the walk can ``hide'' in a component of~$\Z^d\setminus\scrC_{\infty,\alpha}$:

\begin{lemma}[Hidding time estimate]
\label{L:hiding-time}
For~$x\in\Z^d$, let~$\scrG_x=\scrG_x(\omega)$ be as in Proposition~\ref{P:perc}. Set~$c_2:=4d\alpha^{-1}$. Then for all~$\omega\in(0,1]^{\B^d}$,
\begin{equation}
E^x_\omega(T_1)\leq c_2\vert \scrG_x\vert.
\end{equation}
\end{lemma}

\begin{proofsect}{Proof}
This is a restatement of Lemma~3.8 from \cite{BBHK}.
\end{proofsect}

A fundamental concept in the study of random walks in random environments is the ``point of view of the particle.'' The idea is that instead of recording the position of the walk relative to a given environment, we follow the sequence of environments that the walker sees along its path. Explicitly, let~$\tau_x$ denote the ``shift by~$x$'' which is formally defined by
\begin{equation}
(\tau_x\omega)_{yz}:=\omega_{y+x,z+x},\qquad x\in\Z^d,\,(y,z)\in\B^d.
\end{equation}
Given a trajectory $\hat X=(\hat X_n)_{n\ge0}$ of the coarse-grained random walk in environment~$\omega$ with $0\in\scrC_{\infty,\alpha}(\omega)$, the sequence $(\tau_{\hat X_n}\omega)_{n\ge0}$ is itself a Markov chain on the space of environments with stationary measure
\begin{equation}
\Q_\alpha(-):=\Q(-|0\in\scrC_{\infty,\alpha}),
\end{equation}
where
\begin{equation}
\Q(\textd\omega):=\frac{\pi_\omega(0)}{Z}\PP(\textd\omega)\quad\text{for}\quad Z:=\E\pi_\omega(0).
\end{equation}
Furthermore, since~$\PP$ is ergodic with respect to $(\tau_x)_{x\in\Z^d}$, abstract considerations (cf~\cite[Section~3]{BB}) imply that~$\Q_\alpha$ is ergodic with respect to the Markov shift $\omega\mapsto\tau_{\hat X_1}\omega$, where~$\hat X_1$ is sampled from~$P_\omega^0$ and~$\omega$ from~$\Q_\alpha$. Introduce also the shorthand
\begin{equation}
\PP_\alpha(-):=\PP(-|0\in\scrC_{\infty,\alpha})
\end{equation}
and note that~$\Q_\alpha\sim\PP_\alpha$ for all~$\alpha\in(0,\alpha_0]$.
A direct consequence of these constructions and Lemma~\ref{L:hiding-time} is that the time scales of the walk~$X$ and the walk~$\hat X$ are commensurate:

\begin{lemma}
\label{L:time-change} 
For each $\alpha\in(0,\alpha_0]$ there is $\beta=\beta(d,\alpha)\in(0,\infty)$ such that for $\PP_\alpha$-a.e.~$\omega$,
\begin{equation}
\label{E:3.8}
P^0_\omega \biggl(\,\sum^{n}_{\ell=1} T_\ell>\beta n \biggr)\xrightarrow[n\to\infty]{}0.
\end{equation} 
\end{lemma}

\begin{proofsect}{Proof}
Fix~$\alpha\in(0,\alpha_0]$ and let $\beta$ be such that
\begin{equation}
\label{E:beta-choice}
\beta> E_{\Q_\alpha}\bigl(E_\omega^0(T_1)\bigr).
\end{equation}
Such a choice is possible because the expectation on the right is finite by Lemma~\ref{L:hiding-time}, the bounds $\pi_\omega(0)\le2d$ and $\PP(0\in\scrC_{\infty,\alpha})>0$ and the fact that $\E|\scrG_x|<\infty$, as implied by Proposition~\ref{P:perc}(4).
The ergodicity of the Markov shift on the space of environments implies that, for $\PP_\alpha$-a.e.~$\omega$,
\begin{equation}
\frac1n\sum^{n}_{\ell=1} T_\ell\xrightarrow[n\to\infty]{} E_{\Q_\alpha}\bigl(E_\omega^0(T_1)\bigr),\qquad P_\omega^0\text{-a.s.}
\end{equation}
The right-hand side is strictly less than~$\beta$ and so the claim follows.
\end{proofsect}

As alluded to before, the reduction to the coarse-grained walk, and the resulting time change, will need to be performed inside the expectation of random variables
\begin{equation}
R_{n,k}:=\sum_{\ell=t_k}^{2t_k}\1_{\AA_n(\hat X_\ell)}\1_{\{\hat X_\ell\in B_k^\circ\}},
\end{equation}
which --- as we will demonstrate soon --- will serve as a lower bound on the sum in \eqref{keyeq}. We will need estimates on the first two moments of $R_{n,k}$:

\begin{proposition}[Moment bounds]
\label{prop-moments}
Let $d\ge4$ and suppose~$\rho_n$ obeys \eqref{rho-n-decay} for some sequence $\{n_\ell\}_{\ell\ge1}$. Let $\alpha\in(0,\alpha_0]$. Then there are $\PP_\alpha$-a.s.\ finite and positive random variables $C_3=C_3(\omega)$, $C_4=C_4(\omega)$ and $N_2=N_2(\omega)$ such that for all $n\in\{n_\ell\}_{\ell\ge1}$ with~$n\ge N_2$ and all $k$ satisfying \eqref{E:2.24} we have
\begin{equation}
\label{E:3.13}
E_\omega^0\bigl(R_{n,k})\ge C_3(\omega)\rho_n t_k
\end{equation}
and
\begin{equation}
\label{E:3.14}
E_\omega^0\bigl(R_{n,k}^2)\le C_4(\omega)\bigl(\rho_n t_k\bigr)^2.
\end{equation}
\end{proposition}

The proof of these bounds is deferred to Sections~\ref{sec4}-\ref{sec7}. We will now show how the ingredients assemble in the proof of the Key Lemma:

\begin{proofsect}{Proof of Key Lemma from Proposition~\ref{prop-moments}}
It is clear that it suffices to prove the statement for $\PP_\alpha$-a.e.\ $\omega$ and all $\alpha>0$ sufficiently small, because the support of~$\PP$ can be covered by the union of supports of $\PP_{\alpha_r}$ for some $\alpha_r\downarrow0$. We will assume throughout that $n\in\{n_\ell\}_{\ell\ge1}$.

Let~$\alpha\in(0,\alpha_0]$ and let $\omega$ be such that there is a unique infinite connected component~$\scrC_{\infty,\alpha}$ whose complement has only finite connected components.
Let $\beta=\beta(\alpha,d)$ be as in Lemma~\ref{L:time-change} and suppose \eqref{E:3.8} is valid for this~$\omega$. Assume also that $C_3(\omega)$, $C_4(\omega)$ and $N_2(\omega)$ from Proposition~\ref{prop-moments} are finite and positive. Consider the event
\begin{equation}
\EE_k:=\biggl\{\,\sum^{2t_k}_{\ell=1} T_\ell\leq 2\beta t_k \biggr\}.
\end{equation}
Now define~$C_2$ and~$N_1$ as follows: Let $C_2(\omega):=\frac12 C_3(\omega)$ and let $N_1(\omega)$ denote the least integer~$n'\ge N_2(\omega)\vee\texte^{8\beta}$ such that 
\begin{equation}
\label{E:C3}
t_k\ge\texte^{(\log\log n')^2}\quad\Rightarrow\quad
C_3(\omega)\ge4\sqrt{C_4(\omega)P_\omega^0(\EE_k^\cc)}.
\end{equation}
Clearly, $N_1(\omega)<\infty$ because $P_\omega^0(\EE_k^\cc)\to0$ as $k\to\infty$ holds for~$\omega$.

Having made the necessary definitions, we can now get to the actual argument. A starting point is to notice that, for the paths of the random walk $X$ belonging to $\EE_k$ and~$k$ such that $2\beta t_k\le\ffrac n2-1$, the sum in \eqref{keyeq} can be bounded below by $R_{n,k}$,
\begin{equation}
\1_{\EE_k}\sum^{n/2-1}_{\ell=0}\1_{\{X_\ell\in B_k^\circ\}}\1_{\AA_n(X_\ell)}
\ge R_{n,k}\1_{\EE_k}.
\end{equation}
The upper bound in \eqref{E:2.24} shows that $2\beta t_k\le\ffrac n2-1$ once $n\ge\texte^{8\beta}$, and so for $n\ge N_2$ it suffices to derive the desired lower bound for $E_\omega^0(R_{n,k}\1_{\EE_k})$ instead. For this we introduce
\begin{equation}
\FF_{n,k}:=\bigl\{R_{n,k}\le M\rho_n t_k\bigr\},
\end{equation}
where $M>0$ is a number to be determined momentarily, and write
\begin{equation}
\label{E:3.18}
\begin{aligned}
E_\omega^0(R_{n,k}\1_{\EE_k})&\ge E_\omega^0\bigl(R_{n,k}\1_{\EE_{k}\cap\FF_{n,k}}\bigr)
\\
&= E_\omega^0(R_{n,k}) - E_\omega^0(R_{n,k}\1_{\EE_k^\cc\cap\FF_{n,k}}) - E_\omega^0\bigl(R_{n,k}\1_{\FF_{n,k}^\cc}\bigr),
\end{aligned}
\end{equation}
where we also used that $R_{n,k}\ge0$. 

It remains to estimate the three terms on the right-hand side of \eqref{E:3.18}. From the definition of~$\FF_{n,k}$ we immediately have
\begin{equation}
E_\omega^0(R_{n,k}\1_{\EE_k^\cc\cap\FF_{n,k}})\le M\rho_n t_k\, P_\omega^0(\EE_k^\cc).
\end{equation}
For the last term in \eqref{E:3.18}, since $n\ge N_1(\omega)\ge N_2(\omega)$ and~$k$ obeys \eqref{E:2.24}, the Markov inequality and Proposition~\ref{prop-moments} yield
\begin{equation}
E_\omega^0\bigl(R_{n,k}\1_{\FF_{n,k}^\cc}\bigr)\le\frac1{M\rho_n t_k}E_\omega^0\bigl(R_{n,k}^2\bigr)
\le\frac{C_4(\omega)}{M}\rho_n t_k.
\end{equation}
Along with \eqref{E:3.13} this shows that all three terms on the right-hand side of \eqref{E:3.18} are of the same order. This permits us to turn \eqref{E:3.18} into
\begin{equation}
E_\omega^0(R_{n,k}\1_{\EE_k})\ge \Bigl(C_3(\omega)-MP_\omega^0(\EE_k^\cc)-\frac{C_4(\omega)}M\Bigr)\rho_n t_k.
\end{equation}
Now set $M:=[C_4(\omega)/P_\omega^0(\EE_k^\cc)]^{\ffrac12}$ and note that, by \eqref{E:C3} and our choice of $C_2(\omega)$, the term in the parenthesis multiplying $\rho_n t_k$ is at least~$C_2(\omega)$. 
\end{proofsect}

\section{Moment bounds on~$R_{n,k}$}
\label{sec4}\noindent
At this point, the proof of our main result has been reduced to the moment estimates from Proposition~\ref{prop-moments}. There are generally two types of technical ingredients we will need to invoke in both cases: appropriate heat-kernel bounds and estimates on the density of points $x\in B_k^\circ\cap\scrC_{\infty,\alpha}$ where~$\AA_n(x)$ occurs.  
To demonstrate the underlying reason for invoking these facts, let us again begin by a heuristic argument that explains why the bound on~$E_\omega^0(R_{n,k})$ should hold true. 

Consider the coarse-grained walk~$\hat X$ and let $\hat{\cmss P}_\omega$ denote its transition probability on~$\scrC_{\infty,\alpha}$. Explicitly, using the notation \eqref{E:Ts} we have:
\begin{equation}
\hat{\cmss P}_\omega(x,y)=P_\omega^x(X_{T_1}=y),\qquad x,y\in\scrC_{\infty,\alpha}(\omega).
\end{equation}
Then we can write
\begin{equation}
\label{E:ERk-rewrite}
E_\omega^0(R_{n,k})=\sum_{x\in B_k^\circ\cap\scrC_{\infty,\alpha}}\biggl(\,\sum_{\ell=t_k}^{2t_k}
\hat{\cmss P}_\omega^\ell(0,x)\biggr)\,\1_{\AA_n(x)}.
\end{equation}
Since the temporal scale~$t_k$ and the spatial scale of~$B_k$ are related by diffusive scaling, and the chain~$\hat X$ has good mixing properties, it is now \emph{quite} reasonable to expect that~$\hat X_\ell$ is in the time range $t_k\le\ell\le2t_k$ more or less evenly distributed over~$B_k^\circ\cap\scrC_{\infty,\alpha}$. In particular, the sum over~$\ell$ in \eqref{E:ERk-rewrite} is at least of order $t_k^{1-d/2}$, uniformly in $x\in B_k^\circ\cap\scrC_{\infty,\alpha}$. The bound \eqref{E:3.13} is thus reduced to estimating the lower density of $\AA_n$ in~$B_k^\circ\cap\scrC_{\infty,\alpha}$.

Unfortunately, the desired lower bound on $\hat{\cmss P}_\omega^\ell(0,x)$ does not seem to be presently available in the literature and so we will have to state and prove it here:

\begin{lemma}
\label{L:HK-lower}
Let~$d\ge2$. For each~$\alpha\in(0,\alpha_0]$ there is a constant~$c_3>0$ and a $\PP_\alpha$-a.s.\ finite random variable $N_4=N_4(\omega)$ such that
\begin{equation}
\label{E:HK-lower}
\sum_{\ell=t_k}^{2t_k}\hat{\cmss P}_\omega^\ell(0,x)\ge c_3t_k^{1-d/2}
\end{equation}
holds for all $x\in B_k\cap\scrC_{\infty,\alpha}$ whenever $t_k\ge N_4(\omega)$.
\end{lemma}

(We note that our proof of this lemma produces directly a bound on the sum, not on the individual terms.) As already alluded to above, we will need to combine this with the following bound on the density of occurrences of~$\AA_n$ in the set $B_k\cap\scrC_{\infty,\alpha}$:

\begin{lemma}
\label{L:density-lower}
Let~$d\ge4$ and suppose that \eqref{rho-n-decay} holds for some sequence $\{n_\ell\}_{\ell\ge1}$.
Let $\alpha\in(0,\alpha_0]$. There is a constant~$c_4=c_4(d,\alpha)<\infty$ and a $\PP_\alpha$-a.s.\ finite random variable $N_5=N_5(\omega)$ such that for all~$n\in\{n_\ell\}_{\ell\ge1}$ with~$n\ge N_5(\omega)$ and all~$k$ with $t_k\ge\log n$, 
\begin{equation}
\label{E:density-lower}
\sum_{x\in B_k^\circ\cap\scrC_{\infty,\alpha}}\1_{\AA_n(x)}\ge c_4\rho_n|B_k|.
\end{equation}
\end{lemma}

\smallskip
Deferring the proof of these lemmas to the next sections, we observe that the bound on the first moment of~$R_{n,k}$ is now reduced to two lines:

\begin{proofsect}{Proof of \eqref{E:3.13} from Lemma~\ref{L:density-lower}}
The rewrite \eqref{E:ERk-rewrite} and the bounds \eqref{E:HK-lower} and \eqref{E:density-lower} imply the desired estimate with~$C_3(\omega):=c_3c_4\inf_k|B_k|t_k^{-d/2}$ and, e.g., $N_2(\omega):=\texte^{N_4(\omega)}\vee N_5(\omega)$.
\end{proofsect}

Next we turn our attention to the second moment of~$R_{n,k}$. It is not unreasonable to expect that here we will need some form of \emph{upper} bounds on the heat kernel and \emph{upper} bounds on the density of vertices where~$\AA_n$ occurs. Some version of the former is already available:

\begin{lemma}
\label{L:HK-upper}
Let~$d\ge2$. For each~$\alpha\in(0,\alpha_0]$, there is a $\PP_\alpha$-a.s.~finite random variable $C_6=C_6(\omega)$ such that for $\PP_\alpha$-a.e.~$\omega$, 
\begin{equation}
\label{E:HK-upper}
\sup_{x\in\scrC_{\infty,\alpha}(\omega)}\,\hat{\cmss P}_\omega^\ell(0,x)\le\frac{C_6(\omega)}{\ell^{d/2}},\qquad \ell\ge1.
\end{equation}
\end{lemma}

\begin{proofsect}{Proof}
This is a restatement of Lemma~3.2 from~\cite{BBHK}. 
\end{proofsect}

We will need to boost this into an estimate on the Green's function associated with random walk~$\hat X$. For $x,y\in\scrC_{\infty,\alpha}$ this function is defined by
\begin{equation}
\hat{\cmss G}_\omega(x,y):=\sum_{\ell\ge0}\hat{\cmss P}_\omega^\ell(x,y)
=(1-\hat{\cmss P}_\omega)^{-1}(x,y).
\end{equation}
In order to ease the notation, for any~$\omega$ and any $f,g\colon\Z^d\to\R$ with finite supports, let
\begin{equation}
\langle f,g\rangle_\omega:=\sum_{x\in\scrC_{\infty,\alpha}}f(x)g(x)
\end{equation}
denote the inner product with respect to the counting measure on~$\scrC_{\infty,\alpha}$. (A more natural inner product to consider would be that with respect to measure $\pi_\omega$ restricted to~$\scrC_{\infty,\alpha}$. However, the above is what naturally comes up in our calculations; conversions to other inner products will be the subject of Lemma~\ref{L:G-compare}.)  We will then need:

\begin{lemma}
\label{L:G-upper}
Let $d\ge4$ and~$\alpha\in(0,\alpha_0]$. There are~$c_5<\infty$, $\eta<\infty$ and a $\PP_\alpha$-a.s.\ finite random variable $N_5=N_5(\omega)$ such that for all~$n\ge N_5(\omega)$ and all~$k$ with $t_k\ge\log n$, the function
\begin{equation}
\label{E:fk}
f_k(x):=\1_{\AA_n(x)}\1_{\{x\in B_k^\circ\cap\scrC_{\infty,\alpha}\}}
\end{equation}
obeys
\begin{equation}
\bigl\langle f_k,\hat{\cmss G}_\omega f_k\bigr\rangle_\omega\le
c_5\biggl\{\,(\log n)^\eta\,t_k^{d/2}+
\sum_{\begin{subarray}{c}
x,y\in B_k^\circ\cap\scrC_{\infty,\alpha}\\|x-y|\ge\log n
\end{subarray}}
\frac{\1_{\AA_n(x)}\1_{\AA_n(y)}}{1+|x-y|^{d-2}}\biggr\}.
\end{equation}
\end{lemma}

The proof of Lemma~\ref{L:G-upper} will require some non-trivial manipulations with \emph{off-diagonal} heat-kernel bounds and is therefore also deferred to the next sections. In order to estimate the sum on the right-hand side, we will also need to prove:

\begin{lemma}
\label{lemma-sum-upper}
Let~$d\ge3$ and suppose that \eqref{rho-n-decay} holds for some sequence $\{n_\ell\}_{\ell\ge1}$. Then there is a constant~$c_6<\infty$ and a $\PP$-a.s.\ finite random variable~$N_6=N_6(\omega)$ such that, for all $n\in\{n_\ell\}_{\ell\ge1}$ and all~$k$ with~$t_k\ge\log n$,
\begin{equation}
\sum_{\begin{subarray}{c}
x,y\in B_k^\circ\cap\scrC_{\infty,\alpha}\\|x-y|\ge\log n
\end{subarray}}\frac{\1_{\AA_n(x)}\1_{\AA_n(y)}}{1+|x-y|^{d-2}}
\le c_6\rho_n^2t_k^{1+d/2}.
\end{equation}
\end{lemma}

As we will see in the next section, this will be easy to prove once we have a uniform bound on the density of~$\AA_n$ in large rectangular subsets of~$B_k$. We will now show how these ingredients combine into the upper bound on $E_\omega^0(R_{n,k}^2)$:

\begin{proofsect}{Proof of \eqref{E:3.14} from Lemmas~\ref{L:G-upper}--\ref{lemma-sum-upper}}
Throughout, let us assume that $n\in\{n_\ell\}_{\ell\ge1}$. Note that, by \eqref{E:2.24}, the condition $t_k\ge\log n$ from Lemmas~\ref{L:G-upper}--\ref{lemma-sum-upper} is satisfied. Recall \eqref{E:fk}. Writing~$R_{n,k}^2$ as the sum of $f_k(X_\ell)f_k(X_{\ell'})$ over pairs $\ell,\ell'$ with $t_k\le\ell,\ell'\le 2t_k$, the positivity of all terms permits us to estimate the sum as twice the same sum with~$\ell,\ell'$ now obeying $t_k\le\ell\le\ell'\le2t_k$. Applying the Markov property and reparametrizing by means of~$s:=\ell'-\ell$ yields
\begin{equation}
E_\omega^0\bigl(R_{n,k}^2)\le2\sum_{\ell\ge t_k}\sum_{s\ge0}\,\sum_{x,y\in\scrC_{\infty,\alpha}} \hat{\cmss P}_\omega^\ell(0,x)
\hat{\cmss P}_\omega^{s}(x,y)f_k(x)f_k(y),
\end{equation}
where we also extended the summation ranges of~$\ell$ and $s$ to infinity.
Plugging \eqref{E:HK-upper} for $\hat{\cmss P}_\omega^\ell(0,x)$, the sum over~$\ell$ can be estimated by an integral with the result
\begin{equation}
\label{E:3.16}
E_\omega^0\bigl(R_{n,k}^2)\le 2C_6(\omega)\frac{2}{d-2}\,(t_k-1)^{1-d/2}\bigl\langle f_k,\hat{\cmss G}_\omega f_k\bigr\rangle_\omega.
\end{equation}
Lemmas~\ref{L:G-upper} and \ref{lemma-sum-upper} now tell us that, for $n\ge N_5(\omega)\vee N_6(\omega)$ and $t_k\ge\log n$, the inner product is bounded by $c_5\rho_n(\log n)^\eta t_k^{d/2}+c_5c_6\rho_n^2t_k^{1+d/2}$. Now, by \eqref{E:2.24} we in fact have $(\log n)^{\eta+2}\le t_k$ for~$n\gg1$ and so, by $\rho_n\log n\ge1$ (as implied by \eqref{rho-n-decay}),
\begin{equation}
(\log n)^\eta t_k^{d/2}\le \rho_n^2 (\log n)^{\eta+2} t_k^{d/2}\le\rho_n^2 t_k^{1+d/2}
\end{equation}
once $n$ exceeds some finite $n_0$. Summarizing, 
\begin{equation}
(t_k-1)^{1-d/2}\bigl\langle f_k,\hat{\cmss G}_\omega f_k\bigr\rangle_\omega\le 2c_5(1+c_6)\,\rho_n^2\,t_k^2
\end{equation}
is valid once $n\ge n_0$ and $t_k$ obeys \eqref{E:2.24}. The desired claim thus follows for the choices $N_2(\omega):=N_3(\omega)\vee N_5(\omega)\vee n_0$ and $C_4(\omega):=2c_5C_6(\omega)[1+c_6]$.
\end{proofsect}

\section{Density estimates}
\label{sec5}\noindent
The goal of this section is to derive the necessary estimates concerning the density of occurrences of event~$\AA_n$ in~$B_k\cap\scrC_{\infty,\alpha}$ and thus establish Lemmas~\ref{L:density-lower} and~\ref{lemma-sum-upper}.  Both of these lemmas  will make use of the following claim:

\begin{lemma}
\label{L:5.1}
For numbers~$\theta_n\in(0,1)$, let $Z_{n,1},Z_{n,2},\dots$ be i.i.d.\ Bernoulli random variables with parameter~$\theta_n$. If~$\{n_k\}_{k\ge1}$ is a sequence with $\theta_{n_k}\log n_k\to\infty$ as~$k\to\infty$, then for any~$\epsilon>0$,
\begin{equation}
\label{E:4.4}
\sum_{n\in\{n_k\colon k\ge1\}}\,\,n^d\!\sum_{m\ge\epsilon(\log n)^2}\PP\Bigl(\,\frac1{\theta_n m}\sum_{j=1}^m Z_{n,j}\not\in\bigl(\ffrac12\,,2\bigr)\Bigr)<\infty.
\end{equation}
\end{lemma}

\begin{proofsect}{Proof}
By the exponential Chebyshev inequality,
\begin{equation}
\PP\Bigl(\,\frac1{\theta_n m}\sum_{j=1}^m Z_{n,j}\not\in\bigl(\ffrac12\,,2\bigr)\Bigr)\le2\texte^{-\zeta m\theta_n},
\end{equation}
where $\zeta:=\min\{3-\texte,\ffrac12-\texte^{-1}\}$. The sum over~$m$ is dominated by its lowest term which, by $\theta_n\log n\to\infty$, tends to zero faster than any polynomial in~$n$.
\end{proofsect}

We begin with the proof of the upper bound which is easier because there one can immediately drop the restriction that the points be contained in the infinite cluster.

\begin{proofsect}{Proof of Lemma~\ref{lemma-sum-upper}}
Suppose~$d\ge3$. Let~$\Lambda_\ell(x)=x+[-\ell,\ell]^d\cap\Z^d$ and abbreviate $\Lambda_\ell:=\Lambda_\ell(0)$. First we claim that, for some $\PP_\alpha$-a.s.\ finite random variable~$N'=N'(\omega)$, 
\begin{equation}
\label{E:supsup}
\sup_{\begin{subarray}{c}
n\in\{n_k\}\\ n\ge N'(\omega)
\end{subarray}}
\,\max_{x\in\Lambda_n}\,\max_{\frac12\log n\le\ell\le n}\,
\frac1{\rho_n'|\Lambda_\ell|}\sum_{z\in\Lambda_\ell(x)}\1_{\AA_n(z)}\le2,
\end{equation}
where~$\rho_n':=\PP(\AA_n(0))$.
To see this, partition~$\Z^d$ into $6^d$-translates of~$(6\Z)^d$ and label these by~$\Z^d_i$, $i=1,\dots,6^d$. Clearly, it suffices to show the above for~$\Lambda_\ell(x)$ replaced by $\Lambda_\ell^i(x):=\Lambda_\ell(x)\cap\Z_i^d$ --- including the normalization --- for each~$i$. Note that the events~$\AA_n(z)$, $z\in\Lambda_\ell^i(x)$, are i.i.d.\ with probability $\rho_n'$. 
Now fix~$n\in\{n_\ell\}_{\ell\ge1}$, set $\theta_n:=\rho_n'$, $m:=|\Lambda_\ell^i|$ and observe that~$m\gg(\log n)^2$ when~$\ell\ge\frac12\log n$. The probability that the maxima over~$x$ and~$\ell$ in \eqref{E:supsup} exceed~$2$ is then bounded by the~$n$-th term in \eqref{E:4.4}. The Borel-Cantelli lemma and \eqref{E:4.4} imply that this will occur only for finitely many~$n\in\{n_\ell\}_{\ell\ge1}$,~$\PP_\alpha$-a.s., thus proving \eqref{E:supsup}.

Now pick~$n\ge N'(\omega)$, use $\D:=\{2^m\colon m\ge0\}$ to denote the set of dyadic integers and consider the sum in the statement of the lemma. 
We have
\begin{equation}
\label{E:5.4}
\sum_{\begin{subarray}{c}
x,y\in B_k\\|x-y|\ge\log n
\end{subarray}}
\frac{\1_{\AA_n(x)}\1_{\AA_n(y)}}{1+|x-y|^{d-2}}
\le\sum_{\begin{subarray}{c}
M\in\D\\\frac12\log n\le M\le\sqrt{t_k}
\end{subarray}}\,\,\,
\sum_{\begin{subarray}{c}
x,y\in B_k\\M\le|x-y|\le2M
\end{subarray}}
\frac{\1_{\AA_n(x)}\1_{\AA_n(y)}}{1+M^{d-2}}.
\end{equation}
For a fixed~$x$, we extend the sum over~$y$ to $y\in\Lambda_{2M}(x)$; since $M\ge\frac12\log n$, the sum of the indicator of ${\AA_n(y)}$ is then less than $2\rho_n'|\Lambda_{2M}|\le2\rho_n'(4M+1)^d$. The summation range of~$x$ can subsequently be extended to $\Lambda_s$ with~$s:=\sqrt{t_k}$, which contains~$B_k$. Invoking \eqref{E:supsup}, this yields
\begin{equation}
\text{r.h.s.~of~\eqref{E:5.4}}
\le
4(\rho_n')^2(2\sqrt{t_k}+1)^d
\sum_{\begin{subarray}{c}
M\in\D\\\frac12\log n\le M\le\sqrt{t_k}
\end{subarray}}\frac{(4M+1)^d}{1+M^{d-2}}.
\end{equation}
It is now easy to check that the right-hand side is of order~$\rho_n^2t_k^{1+d/2}$.
\end{proofsect}

\begin{proofsect}{Proof of Lemma~\ref{L:density-lower}}
Suppose~$d\ge4$. For the lower bound we will invoke some more sophisticated facts about percolation in~$d\ge3$. Let $p:=\PP(\omega_b\ge\alpha)$ and let~$M$ be a dyadic integer such that the bond percolation with parameter~$p$ in the slab 
\begin{equation}
\H_M(\ell):=\{\ell,\ell+1,\dots,\ell+M-1\}\times\Z^{d-1}
\end{equation}
of width~$M$ 
contains an infinite cluster almost surely. The existence of such an $M$ is guaranteed by Grimmett and Marstrand~\cite{Grimmett-Marstrand}. In particular, by the uniqueness of the infinite component in the slab (e.g., via Burton and Keane~\cite{Burton-Keane}) the restriction of~$\scrC_{\infty,\alpha}$ to~$\H_M(\ell)$ contains a unique infinite connected component~$\scrC_{\infty,\alpha}(\ell)$~$\PP$-a.s. Note that~$\scrC_{\infty,\alpha}(\ell)$ is independent of the edges with at least one endpoint outside~$\H_M(\ell)$.

Abbreviate $\mathbb S_M:=\{x=(x_1,\dots,x_d)\in(3\Z)^d\colon x_1\in3M\Z\}$ and let~$\AA_n'(x)$ denote the subset of~$\AA_n(x)$ containing the configurations such that a trap occurs at~$x$ with the trap edge $(y,z)$ such that $y:=x-\hate_1$ and $z:=x-2\hate_1$. (Here $\hate_1:=(1,0,\dots,0)$.) Clearly,
\begin{equation}
\label{E:4.9}
\sum_{x\in B_k^\circ\cap\scrC_{\infty,\alpha}}\1_{\AA_n(x)}
\ge\sum_{\ell\in\Z}\,\sum_{x\in B_k^\circ\cap\mathbb S_M}
\1_{\AA_n'(x)}\1_{\{x\in\scrC_{\infty,\alpha}(3\ell M)\}}.
\end{equation}
A key point of the construction is that, conditional on all infinite clusters $\{\scrC_{\infty,\alpha}(3\ell M)\colon \ell\in\Z\}$, the events $\{\AA_n'(x)\colon x\in\bigcup_{\ell\in\Z}\scrC_{\infty,\alpha}(3\ell M)\cap\mathbb S_M\}$ are i.i.d.\ with probability~$\rho_n':=\PP(\AA'_n(0))$. 

To estimate the right-hand side of \eqref{E:4.9}, let~$\mathscr F_M$ denote the $\sigma$-algebra generated by the restriction of~$\omega$ to the union of slabs $\bigcup_{\ell\in\Z}\H_M(3\ell M)$ and introduce the ($\mathscr F_M$-measurable) quantity
\begin{equation}
Q_k:=\sum_{\ell\in\Z}\,\sum_{x\in B_k^\circ\cap\mathbb S_M}\1_{\{x\in\scrC_{\infty,\alpha}(3\ell M)\}}.
\end{equation}  
Lemma~\ref{L:5.1} and the aforementioned independence yield
\begin{equation}
\sum_{n\in\{n_j\}}\,\sum_{\begin{subarray}{c}
k\colon t_k\ge\log n\\ Q_k\ge(\log n)^2
\end{subarray}}
\PP\biggl(\,\frac1{\rho_n'Q_k}\sum_{\ell\in\Z}\,\sum_{x\in B_k^\circ\cap\mathbb S_M}
\1_{\AA_n'(x)}\1_{\{x\in\scrC_{\infty,\alpha}(3\ell M)\}}\le\frac12\bigg|\mathscr F_M\biggr)<\infty.
\end{equation}
Therefore, in light of the restriction $t_k\ge\log n$, there exists $N_5'=N_5'(\omega)$ such that
\begin{equation}
\begin{aligned}
n\ge N_5',\,n\in\{n_j\},\quad\,\,
\\
t_k\ge\log n,\,Q_k\ge (\log n)^2
\end{aligned}
\qquad\text{imply}\quad
\sum_{x\in B_k^\circ\cap\scrC_{\infty,\alpha}}\1_{\AA_n(x)}
\ge \frac12\rho_n Q_k.
\end{equation}
But the Spatial Ergodic Theorem yields $Q_k/|B_k^\circ|\to\psi\in(0,1)$, where $\psi$ is $(1/3)^{d-1}$ of the (non-random) density of $\scrC_{\infty,\alpha}(0)$ in the hyperplane $\{x\in\Z\colon x_1=0\}$. Hence, there is $N_5''=N_5''(\omega)$ such that $n\ge N_5''$ and $t_k\ge\log n$ forces $Q_k\ge\frac12\psi|B_k|$ and (by $d>2$) also $Q_k\ge (\log n)^2$. Noting that $\rho_n\le2d\rho_n'$, the claim follows with $N_5:=N_5'\vee N_5''$ and $c_4:=\frac14\psi/(2d)$.
\end{proofsect}

\section{Heat-kernel input: upper bound}
\label{sec6}\noindent
Here we establish the first part of the claims involving heat kernel bounds that are needed in the proof of Proposition~\ref{prop-moments}. Specifically, we will give the proof of Lemma~\ref{L:G-upper}. The strategy is to convert this to the same problem for the simple random walk on the supercritical percolation cluster. For this random walk we can apply existing results obtained earlier by Biskup and Prescott~\cite{BP} and Barlow and Hambly~\cite{Barlow-Hambly}. 

Given $\alpha\in(0,\alpha_0]$, let us regard $\scrC_{\infty,\alpha}$ as a graph with edge set inherited from $\{b\in\B^d\colon\omega_b\ge\alpha\}$. Let $\widetilde{\cmss P}_{\alpha,\omega}$ denote the transition probability for the simple random walk on~$\scrC_{\infty,\alpha}(\omega)$ which is the Markov chain for the conductances that are set to one for edges in~$\scrC_{\infty,\alpha}$ and to zero otherwise. Let $\widetilde{\cmss G}_{\alpha,\omega}(x,y):=(1-\widetilde{\cmss P}_{\alpha,\omega})^{-1}(x,y)$ be the Green's function for the transition kernel $\widetilde{\cmss P}_{\alpha,\omega}$.

\begin{lemma}
\label{L:Barlow}
Let~$d\ge3$ and $\alpha\in(0,\alpha_0]$. Then for $\PP_\alpha$-a.e.~$\omega$,
\begin{equation}
\label{E:Barlow}
\widetilde{\cmss G}_{\alpha,\omega}(x,y)\le\frac{\tilde c_1}{|x-y|^{d-2}}\quad\text{if}\quad |x-y|>S_x\wedge S_y,
\end{equation}
where $\{S_x(\omega)\colon x\in\scrC_{\infty,\alpha}\}$ are random variables satisfying
\begin{equation}
\PP\bigl(S_x\ge r\big|x\in\scrC_{\infty,\alpha}\bigr)\le\texte^{-\tilde c_2 r^\delta},\qquad r>0,
\end{equation}
for some constants $\tilde c_1,\tilde c_2,\delta\in(0,\infty)$.
\end{lemma}

\begin{proofsect}{Proof}
This is a restatement of the upper bound from Theorem~1.2 of Barlow and Hambly~\cite{Barlow-Hambly}.
\end{proofsect}

In addition we will need to make the following observation:

\begin{lemma}
\label{L:G-absolute}
Let $d\ge3$ and $\alpha\in(0,\alpha_0]$. Then
\begin{equation}
\E_\alpha \widetilde{\cmss G}_{\omega,\alpha}(0,0)<\infty.
\end{equation}
\end{lemma}

\begin{proofsect}{Proof}
We will plug into explicit expressions derived in~\cite{BP}. Let us use
\begin{equation}
\label{E:degree}
d_\omega(x):=\sum_{y\colon |y-x|=1}\1_{\{\omega_{xy}\ge\alpha\}}
\end{equation}
to denote the degree of~$x$ in the graph~$\scrC_{\infty,\alpha}(\omega)$. Let
\begin{equation}
\widetilde{\cmss q}_t(x,y):=\frac1{d_\omega(y)}\sum_{n\ge0}\,\frac{t^n}{n!}\texte^{-t}\,\widetilde{\cmss P}_\omega^n(x,y)
\end{equation}
denote the continuous-time heat kernel associated with the simple random walk on~$\scrC_{\infty,\alpha}.$ We claim that the random variable
\begin{equation}
K_1(\omega):=\sup_{t\ge1}\,t^{d/2}\widetilde{\cmss q}_t(0,0)
\end{equation}
satisfies $\E K_1<\infty$. This is seen as follows: By way of our assumption $c(d)\PP(\omega_b\ge\alpha)<1$ it is not hard to check that the random variable $C_{\text{iso}}(0)^{-1}$ defined in \cite[Eq.\ (6.5)]{BP} has a stretched exponential tail and thus has all positive moments. By \cite[Proposition~6.1]{BP}, $K_1(\omega)$ is bounded by a constant times $C_{\text{iso}}(0)^{-d}$ and so $K_1(\omega)$ has all moments as well. The observation
\begin{equation}
\widetilde{\cmss G}_{\omega,\alpha}(0,0)=d_\omega(0)\int_0^\infty\widetilde{\cmss q}_t(0,0)\,\textd t
\le 1+2dK_1(\omega)\int_1^\infty t^{-d/2}\textd t
\end{equation}
now proves the claim.
\end{proofsect}

We will use this in conjunction with the following comparison statement:

\begin{lemma}
\label{L:G-compare}
Suppose $d\ge3$ and let $\alpha\in(0,\alpha_0]$. Then for $\PP_\alpha$-a.e.~$\omega$ and any (positive!) function $f\colon\scrC_{\infty,\alpha}(\omega)\to[0,\infty)$ with finite support,
\begin{equation}
\label{E:Green-compare}
\bigl\langle f,\hat{\cmss G}_\omega f\bigr\rangle_\omega\le\Bigl(\frac{2d}\alpha\Bigr)^2\bigl\langle f,\widetilde{\cmss G}_{\alpha,\omega} f\bigr\rangle_\omega.
\end{equation}
\end{lemma}

\begin{proofsect}{Proof}
Informally, the comparison between the quadratic forms in \eqref{E:Green-compare} is a consequence of the fact that the Green's function, being the inverse of the generator of the Markov chain, is non-increasing, as an operator, in the conductances. Nonetheless, as the stationary measures for the two chains are different, the only way we can see how to employ this observation rigorously is by way of introducing an intermediate Markov chain.

Recall our notation $d_\omega(x)$ for the degree of~$x$ in~$\scrC_{\infty,\alpha}(\omega)$ and consider the following transition kernel on $\scrC_{\infty,\alpha}$:
\begin{equation}
\label{bar-P}
\overline{\cmss P}_\omega(x,y):=
\begin{cases}
\pi_\omega(x)^{-1}\alpha\1_{\{\omega_{xy}\ge\alpha\}},\qquad&\text{if }|x-y|=1,
\\*[1mm]
1-\pi_\omega(x)^{-1}\alpha d_\omega(x),\qquad&\text{if }x=y,
\\*[1mm]
0,\qquad&\text{otherwise}.
\end{cases}
\end{equation}
The distinction compared to the Markov chain described by $\widetilde{\cmss P}_\omega$ is that this chain is delayed at each~$x$ for a time that is geometrically distributed with parameter $\alpha d_\omega(x)\pi_\omega(x)^{-1}$. As the Green's function counts the expected number of visits to a given point, the Green's function $\overline{\cmss G}_\omega$ corresponding to $\overline{\cmss P}_\omega$ satisfies
\begin{equation}
\label{E:6.6}
\overline{\cmss G}_\omega(x,y)=\widetilde{\cmss G}_{\alpha,\omega}(x,y)\,\frac{\pi_\omega(y)}{\alpha d_\omega(y)}.
\end{equation}
(This can be also checked directly from $(1-\overline{\cmss P}_\omega)(x,y)=\alpha d_\omega(x)\pi_\omega(x)^{-1}(1-\widetilde{\cmss P}_\omega)(x,y)$ as implied by \eqref{bar-P}.)
The reason for consideration of $\overline{\cmss P}_\omega$ is that, unlike $\widetilde{\cmss P}_\omega$, this chain is stationary and reversible with respect to~$\pi_\omega$. As a consequence of the easy operator bound
\begin{equation}
\label{E:6.11a}
1-\hat{\cmss P}_\omega\ge1-\overline{\cmss P}_\omega\qquad\text{ on }\ell^2(\scrC_{\infty,\alpha},\pi_\omega),
\end{equation}
we thus have $\hat{\cmss G}_\omega\le\overline{\cmss G}_\omega$ on $\ell^2(\scrC_{\infty,\alpha},\pi_\omega)$. Using $\langle f,g\rangle_{\pi_\omega}$ to abbreviate the canonical inner product in $\ell^2(\scrC_{\infty,\alpha},\pi_\omega)$, for positive functions $f$ we now get
\begin{multline}
\qquad\quad
\bigl\langle f,\hat{\cmss G}_\omega f\bigr\rangle_\omega
\le\frac1\alpha\bigl\langle f,\hat{\cmss G}_\omega f\bigr\rangle_{\pi_\omega}
\le\frac1\alpha\bigl\langle f,\overline{\cmss G}_\omega f\bigr\rangle_{\pi_\omega}
\\
\le\frac{2d}{\alpha^2}\bigl\langle f,\widetilde{\cmss G}_{\alpha,\omega} f\bigr\rangle_{\pi_\omega}
\le\Bigl(\frac{2d}\alpha\Bigr)^2\bigl\langle f,\widetilde{\cmss G}_{\alpha,\omega} f\bigr\rangle_{\omega},
\qquad\quad
\end{multline}
where we used $\alpha\le\pi_\omega\le2d$, $d_\omega(x)\ge1$ and \eqref{E:6.6} to get the first, third and last inequalities and \eqref{E:6.11a} to get the second inequality.
\end{proofsect}

We are now ready to assemble the ingredients in the upper bound on $E_\omega^0(R_{n,k}^2)$:

\begin{proofsect}{Proof of Lemma~\ref{L:G-upper}}
First we note that the random variables $(S_x)$ from Lemma~\ref{L:Barlow} satisfy an a.s.\ estimate. Indeed, a Borel-Cantelli argument shows that for any $\theta>1/\delta$ there is a $\PP_\alpha$-a.s.~finite random variable $K=K(\omega)$ such that
\begin{equation}
\max_{x\in B_k^\circ\cap\scrC_{\infty,\alpha}}S_x(\omega)\le [\log t_k]^\theta,\qquad t_k\ge K(\omega).
\end{equation}
Assuming (without loss of generality)~$\theta>1$ and substituting \eqref{E:Barlow} when $|x-y|\ge [\log n]^\theta$ into the definition of $\widetilde{\cmss G}_{\alpha,\omega}(x,y)$ we thus get
\begin{equation}
\label{E:6.14}
\sum_{\begin{subarray}{c}
x,y\\|x-y|\ge[\log n]^\theta
\end{subarray}}
f_k(x) f_k(y)\widetilde{\cmss G}_{\alpha,\omega}(x,y)\le
\sum_{\begin{subarray}{c}
x,y\in B_k^\circ\cap\scrC_{\infty,\alpha}\\|x-y|\ge\log n
\end{subarray}}
c_6\frac{\1_{\AA_n(x)}\1_{\AA_n(y)}}{1+|x-y|^{d-2}}.
\end{equation}
for some absolute constant $c_6=c_6(d)$. For the pairs $(x,y)$ with $|x-y|\le[\log n]^\theta$, here we drop the indicators of $\AA_n(x)$ and $\AA_n(y)$ and invoke the standard fact
\begin{equation}
\widetilde{\cmss G}_{\alpha,\omega}(x,y)\le \widetilde{\cmss G}_{\alpha,\omega}(x,x)^{\ffrac12}\,\widetilde{\cmss G}_{\alpha,\omega}(y,y)^{\ffrac12}.
\end{equation}
Now we apply Cauchy-Schwarz (still under the restriction $|x-y|\le[\log n]^\theta$) to get 
\begin{equation}
\label{E:6.16}
\sum_{\begin{subarray}{c}
x,y\\|x-y|\le[\log n]^\theta
\end{subarray}}
f_k(x) f_k(y)\widetilde{\cmss G}_{\alpha,\omega}(x,y)\le
\sum_{\begin{subarray}{c}
x,y\in B_k^\circ\cap\scrC_{\infty,\alpha}\\|x-y|\le[\log n]^\theta
\end{subarray}}
\widetilde{\cmss G}_{\alpha,\omega}(x,x).
\end{equation}
Summing over~$y$ yields a multiplicative term of order $[\log n]^{d\theta}$ defining $\eta:=d\theta$. The sum over~$x$ is then estimated using the Pointwise Spatial Ergodic Theorem and the bound on the Green's function from Lemma~\ref{L:G-absolute} by a constant times~$t_k^{d/2}$, provided~$t_k$ is sufficiently large. Combining \eqref{E:6.14} and \eqref{E:6.16}, the claim follows.
\end{proofsect}

\section{Heat-kernel input: lower bound}
\label{sec7}\noindent
Our final task in this paper is to establish the lower bound in Lemma~\ref{L:HK-lower}. Unable to directly plug into estimates that exist  in the literature, we will have to reproduce the corresponding argument leading to Proposition~5.1 of Barlow~\cite{Barlow} which is itself based on ideas adapted from Fabes and Stroock~\cite{Fabes-Stroock} and Nash~\cite{Nash}. A slight drawback of this route is that we have to work with the continuous time version of the chain~$\hat X$. 

Fix~$\omega$ with~$0\in\scrC_{\infty,\alpha}$ and consider the (constant-speed) Markov process~$\widetilde X=(\widetilde X_t)_{t\ge0}$ on~$\scrC_{\infty,\alpha}$ with generator~$\LL_{\alpha,\omega}$ that is defined by
\begin{equation}
(\LL_{\alpha,\omega} f)(x):=\sum_{y\in\scrC_{\infty,\alpha}}\hat{\cmss P}_\omega(x,y)\bigl[\,f(y)-f(x)\bigr],\qquad x\in\scrC_{\infty,\alpha}.
\end{equation}
Alternatively, $\widetilde X_t:=\hat X_{N_t}$, where $N_t$ is the rate-one Poisson process at time~$t$.
Let ${\cmss q}_t(x,y)$ denote the associated heat kernel, 
\begin{equation}
{\cmss q}_t(x,y):=\frac{P_\omega^x(\widetilde X_t=y)}{\pi_\omega(y)},\qquad x,y\in\scrC_{\infty,\alpha},
\end{equation}
where, abusing the notation slightly, $P_\omega^x$ denotes the law of~$\widetilde X$ with~$P_\omega^x(\widetilde X_0=x)=1$. (The normalization ensures ${\cmss q}_t(x,y)={\cmss q}_t(y,x)$.) We will need the following estimate:

\begin{proposition}
\label{prop-lower}
Let~$d\ge2$ and $\alpha\in(0,\alpha_0]$.
There are constants~$c_7>0$ and $\xi>0$ and a $\PP_\alpha$-a.s.\ finite random variable $R_0=R_0(\omega)$ such that for  $\PP_\alpha$-a.e.~$\omega$ and all all $R\ge R_0(\omega)$,
\begin{equation}
\label{E:cont-time-HKLB}
\min_{\begin{subarray}{c}
x\in\scrC_{\infty,\alpha}\\|x|\le R
\end{subarray}}
{\cmss q}_t(0,x)\ge c_7\texte^{-\xi tR^{-2}}\,\,R^{-d},\qquad t\ge R^2.
\end{equation}
\end{proposition}

Before we delve into the proof of this claim, let us see how it implies Lemma~\ref{L:HK-lower}:

\begin{proofsect}{Proof of Lemma~\ref{L:HK-lower}}
Consider the quantities
\begin{equation}
a_{n,k}:=\int_{\frac43t_k}^{\frac53t_k}\textd t\,\texte^{-t}\,\frac{t^n}{n!}.
\end{equation}
Then, since $\widetilde X_t$ has the law of $\hat X_n$ at $n:=\text{Poisson}(t)$,
\begin{equation}
\pi_\omega(x)\int_{\frac43t_k}^{\frac53t_k}\textd t\,{\cmss q}_t(0,x)=\sum_{n\ge0}\hat{\cmss P}_\omega^n(0,x)a_{n,k}.
\end{equation}
By Proposition~\ref{prop-lower} with $R:=\sqrt{t_k}$, the integral on the left is order $t_k^{1-d/2}$ uniformly in $x\in B_k\cap\scrC_{\infty,\alpha}$. Since~$a_{n,k}\le1$, $\hat{\cmss P}_\omega^n(0,x)\le1$ and $\pi_\omega(x)\ge\alpha$, it suffices to show
\begin{equation}
\sum_{n\not\in[t_k,2t_k]}a_{n,k}\le\texte^{-ct_k},\quad k\ge1,
\end{equation}
for some~$c>0$. Let~$Z_1,Z_2,\dots$ be i.i.d.\ exponential with parameter one. Then~$a_{n,k}$ is the probability that $Y_n:=Z_1+\dots+Z_{n+1}\in[\frac43t_k,\frac53t_k]$. But the mean of~$Y_n$ is~$n+1$ and~$Z_1$ has exponential moments. So, by Cram\'er's theorem (cf, e.g., den Hollander \cite[Theorem~1.4]{denHolla}) this probability is exponentially small in the distance of~$n$ to $[\frac43t_k,\frac53t_k]$, which is at least $t_k/3$.
\end{proofsect}

The remainder of this section will be spent on proving Proposition~\ref{prop-lower}.
In order to appreciate better the forthcoming definitions, it is instructive to check how the desired lower bound is derived for continuous diffusions in uniformly elliptic environments --- i.e., diffusions on~$\R^d$ with generator $(Lf)(x,t):=\sum_{i,j}\partial_i(a_{ij}(x,t)\partial_j f)(x,t)$, where~$\partial_i$ is the partial derivative with respect to~$x_i$, $i=1,\dots,d$, and where the coefficients $a=(a_{ij})$ are uniformly elliptic in the sense that, for some~$\lambda\in(0,1)$, all~$x\in\R^d$ and all~$t\ge0$,
\begin{equation}
\lambda^{-1}|\xi|^2\le\sum_{ij=1}^d a_{i,j}(x,t)\,\xi_i\xi_j\le\lambda|\xi|^2,\qquad \xi\in\R^d.
\end{equation}
Here Fabes and Stroock (cf~\cite[Section~2]{Fabes-Stroock}) invoke an argument of Nash~\cite{Nash} that goes as follows: Let $y\mapsto\Gamma_a(t,x,y)$ denote the transition density for the above diffusion started at~$x$ and observed at time~$t$. Setting
\begin{equation}
\label{E:H-int}
H_z(t):=\int\textd y\,\texte^{-\pi|y|^2}\log\Gamma_a(t,z,y)
\end{equation}
one then shows, via a differential inequality for~$t\mapsto H_z(t)$, that $H_z(t)\ge c'(\lambda)$ uniformly for all $t\in[0,1]$, all~$z\in\R^d$ with $|z|\le1$ and all~$a$ as above (cf \cite[Lemma~2.1]{Fabes-Stroock}). Invoking the Chapman-Kolmogorov equations,
\begin{equation}
\Gamma_a(2,0,x)\ge\int\textd z\, \texte^{-\pi|z|^2}\,\Gamma_a(1,0,z)\Gamma_a(1,z,x).
\end{equation}
But~$\Gamma_a(1,z,x)=\Gamma_{\tilde a}(1,0,x-z)$ for $\tilde a$ denoting the map of~$a$ under a linear transformation (shift and reflection) of~$\R^d$ and so by taking logs and applying Jensen's inequality for the probability measure $\texte^{-\pi|z|^2}\textd z$ (cf \cite[Lemma~2.6]{Fabes-Stroock}) we get
\begin{equation}
\log\Gamma_a(2,0,x)\ge2c'(\lambda).
\end{equation}
By virtue of shifts and scaling (recall that~$\Gamma_a$ is a \emph{spatial} density and the heat equation is invariant under the diffusive scaling of space and time), the fact that this holds \emph{uniformly} in~$a$ implies the desired claim $\Gamma_a(2t,x,y)\ge\texte^{2c'(\lambda)}t^{-d/2}$ for~$|x-y|\le\sqrt{t}$.

\smallskip
There are several technical obstacles that prevent a direct application of this argument to our present setting. The three most important ones are as follows: 
\begin{enumerate}
\item[(1)]
Our spatial variables are discrete, so the diffusive scaling cannot be used.
\item[(2)]
Our environment is not uniformly elliptic on all scales, so we have to truncate the integral in \eqref{E:H-int} to ``good'' regions.
\item[(3)]
The derivation of the differential inequality for~$t\mapsto H_z(t)$ in \cite{Fabes-Stroock} does not carry directly over to the discrete setting.
\end{enumerate}
Fortunately, all of these obstacles have already been addressed by Barlow in his derivation of a uniform lower bound on the heat-kernel for the random walk on the supercritical percolation cluster; cf~\cite[Proposition~5.1]{Barlow}. So we just need to adapt Barlow's reasoning while paying special attention only to the steps that require modifications due to a (slightly) more general setting.

\smallskip
Suppose~$\omega$ is such that~$0\in\scrC_{\infty,\alpha}$ and recall that $\dist_\omega(x,y)$ stands for the graph-theoretical distance on~$\scrC_{\infty,\alpha}$ between~$x$ and~$y$ --- i.e., the length of the shortest path from~$x$ to~$y$ over edges with $\omega$-conductance at least~$\alpha$. For~$R\ge1$, let
\begin{equation}
K_R:=\bigl\{x\in\scrC_{\infty,\alpha}\colon\dist_\omega(0,x)\le R\bigr\}.
\end{equation}
Introduce the function
\begin{equation}
\varphi(x):=\biggl(\frac{R\wedge\dist_\omega(x,K_R^\cc)}R\biggr)^2
\end{equation}
and consider the weighted measure~$\nu=\nu_{R,\omega}$ defined by
\begin{equation}
\nu(x):=V_R^{-1}\varphi(x)\pi_\omega(x),
\end{equation}
where $V_R:=\sum_x\varphi(x)\pi_\omega(x)$. This $\nu$ will be the analogue of the probability measure $\texte^{-\pi|y|^2}\textd y$ in the continuum setting. Fix~$z\in K_R$ and abbreviate
\begin{equation}
w_{z,t}(y):=\log\bigl(V_R\,{\cmss q}_t(z,y)\bigr),
\end{equation}
where $y\mapsto V_R\,{\cmss q}_t(z,y)$ is the analogue of the transition density $\Gamma_a$. Finally, let
\begin{equation}
H_z(t):=E_\nu\bigl(w_{z,t}(\cdot)\bigr)
\end{equation}
play the role of the quantity in \eqref{E:H-int}. A starting point of the derivation of a differential inequality for~$t\mapsto H_z(t)$ is the following bound:

\begin{lemma}
\label{L:7.2}
Abbreviate $\hat\omega_{xy}:=\pi_\omega(x)\hat{\cmss P}_\omega(x,y)$. Then for any~$z\in K_R$,
\begin{equation}
\label{E:7.13}
\begin{aligned}
V_R \frac{\textd}{\textd t}H_z(t)\,&\ge
\,\frac14\sum_{x,y\in K_R}\hat\omega_{xy}\bigl(\varphi(x)\wedge\varphi(y)\bigr)\bigl[w_{z,t}(x)-w_{z,t}(y)\bigr]^2
\\
&\qquad-\frac14\sum_{x,y\in K_R}\hat\omega_{xy}\frac{(\varphi(x)-\varphi(y))^2}{\varphi(x)\wedge\varphi(y)}
\\
&\qquad\quad-\frac14\sum_{x\in K_R}\sum_{y\in K_R^\cc}\hat\omega_{xy}\,\varphi(x)\biggl(1-\frac{{\cmss q}_t(z,y)}{{\cmss q}_t(z,x)}\biggr).
\end{aligned}
\end{equation}
\end{lemma}

\begin{proofsect}{Proof}
This is proved by literally following the calculation that begins at the bottom of page~3070 and ends on line (5.9) on page~3071 of~\cite{Barlow}. The fact that we used $\dist_\omega$ instead of (perhaps more natural) $\textd'_\omega$-distance is immaterial for the calculation.
\end{proofsect}
 
Next we will estimate the terms on the right-hand side of \eqref{E:7.13}. Notice that, by \eqref{E:distance-comp}, the graph-theoretical distance $\dist_\omega(0,x)$ and the $\ell_2$-metric $|x|$ are commensurate on~$\scrC_{\infty,\alpha}$. In particular, $K_R$ is contained and contains $\ell_2$-balls of radius of order~$R$ and, by ergodicity of the infinite cluster, $V_R$ thus grows proportionally to~$R^d$ as~$R\to\infty$. 

\begin{lemma}
\label{L:7.3}
There is~$c_8<\infty$ and a $\PP_\alpha$-a.s.~finite random variable~$R_1=R_1(\omega)$ such that
\begin{equation}
\label{E:7.3}
\sum_{x,y\in K_R}\hat\omega_{xy}\frac{(\varphi(x)-\varphi(y))^2}{\varphi(x)\wedge\varphi(y)}
\le c_8 V_R R^{-2},\qquad R\ge R_1(\omega).
\end{equation}
\end{lemma}

\begin{proofsect}{Proof}
Let~$x,y\in K_R$ and set $k:=\dist_\omega(x,K_R^\cc)\wedge \dist_\omega(y,K_R^\cc)$ and $s:=\dist_\omega(x,y)$. Then $|\varphi(x)-\varphi(y)|\le (2ks+s^2) R^{-2}$ and so, since~$k\ge1$,
\begin{equation}
\frac{(\varphi(x)-\varphi(y))^2}{\varphi(x)\wedge\varphi(y)}\le\Bigl(\frac{2ks+s^2}k\Bigr)^2R^{-2}\le 9R^{-2}\dist_\omega(x,y)^4.
\end{equation}
Therefore,
\begin{equation}
\label{E:7.20eq}
\text{l.h.s.\ of~\eqref{E:7.3}}
\le 9 R^{-2} \sum_{x\in K_R}h_\alpha\circ\tau_x(\omega).
\end{equation}
where
\begin{equation}
\label{E:halpha}
h_\alpha(\omega):=\1_{\{0\in\scrC_{\infty,\alpha}\}}
\sum_{z\in\scrC_{\infty,\alpha}}\hat\omega_{0,z}\,\dist_\omega(0,z)^4.
\end{equation}
Bounding $\dist_\omega(0,z)\le\text{diam}_\omega(\scrG_0)$ for $\hat\omega_{0,z}>0$, using that $\text{diam}_\omega(\scrG_0)$ has all moments by Proposition~\ref{P:perc}(4) and noting that the sum over $\hat\omega_{0,z}$ equals $\pi_\omega(0)\le2d$, we have $\E h_\alpha(\omega)<\infty$. Since \eqref{E:distance-comp} permits us to dominate the sum over~$x\in K_R$ by that over a cube of side proportional~$R$, the Spatial Ergodic Theorem shows that the sum in \eqref{E:7.20eq} is bounded by a constant times~$V_R$ once~$R$ exceeds a random quantity~$R_1(\omega)$.
\end{proofsect}

\begin{lemma}
\label{L:7.4}
There is a constant~$c_9<\infty$ and a $\PP_\alpha$-a.s.~finite random variable~$R_2=R_2(\omega)$ such that for all~$z$,
\begin{equation}
\sum_{x\in K_R}\sum_{y\in K_R^\cc}\hat\omega_{xy}\,\varphi(x)\biggl(1-\frac{{\cmss q}_t(z,y)}{{\cmss q}_t(z,x)}\biggr)
\le c_9V_R R^{-2},\qquad R\ge R_2(\omega).
\end{equation}
\end{lemma}

\begin{proofsect}{Proof}
Since~$\dist_\omega(x,K_R^\cc)\le\dist_\omega(x,y)$ whenever~$x\in K_R$ and~$y\in K_R^\cc$, we can dominate
\begin{equation}
\varphi(x)\le R^{-2}\dist_\omega(x,y)^2\le R^{-2}\dist_\omega(x,y)^4
\end{equation}
for each pair~$x,y$ contributing to the sum. Dropping the ratio of the ${\cmss q}_t$-terms, the result is estimated by the right-hand side of \eqref{E:7.20eq}.
\end{proofsect}

\begin{corollary}
\label{cor-monotone}
For~$R\ge R_1(\omega)\vee R_2(\omega)$ and all~$z\in K_R$, the function
\begin{equation}
t\mapsto H_z(t)+\frac14(c_8+c_9)R^{-2}\,t
\end{equation}
is non-decreasing on $[0,\infty)$.
\end{corollary}

\begin{proofsect}{Proof}
Let~$h(t):=H_z(t)+\frac14(c_8+c_9)R^{-2}\,t$ and note that, by Lemmas~\ref{L:7.3}-\ref{L:7.4}, $V_R\,h'(t)$ exceeds the first term on the right-hand side of \eqref{E:7.13}. In particular, $h'(t)\ge0$.
\end{proofsect}

The bounds on the last two terms in \eqref{E:7.13} suggest that perhaps also the first term should be at least of order~$R^{-2}$. This is indeed the case thanks to:

\begin{lemma}[Weighted Poincar\'e inequality]
\label{L:7.6}
For any~$\alpha\in(0,\alpha_0]$, there is a constant $c_{10}=c_{10}(\alpha,d)>0$ and a $\PP_\alpha$-a.s.\ finite random variable~$R_3=R_3(\omega)$ such that
\begin{equation}
\label{}
V_R^{-1}\sum_{x,y}\hat\omega_{xy}\bigl(\varphi(x)\wedge\varphi(y)\bigr)\bigl[\,f(x)-f(y)\bigr]^2
\ge c_{10}\, R^{-2}\,\text{\rm Var}_\nu(f)
\end{equation}
holds for any~$R\ge R_3$ and any function $f\colon\scrC_{\infty,\alpha}(\omega)\to\R$ with support in~$K_R$.
\end{lemma}

\begin{proofsect}{Proof}
We will reduce this to the corresponding statement in \cite[Theorem~4.8]{Barlow}. Consider the collection of conductances $(\widetilde\omega_{xy})$ defined by
\begin{equation}
\widetilde\omega_{xy}:=\1_{|x-y|=1}\1_{\{\omega_{xy}\ge\alpha\}}\1_{\{x\in\scrC_{\infty,\alpha}\}},\quad x,y\in\Z^d,
\end{equation}
and let $\kappa(x):=\widetilde V_R^{-1}\varphi(x)\1_{\{x\in\scrC_{\infty,\alpha}\}}d_\omega(x)$, where we recall the notation \eqref{E:degree} and where~$\widetilde V_R$ is the number that makes~$\kappa$ a probability measure. By Theorem~4.8 and the fact that~$K_R$ is ``very good'' (in the language of~\cite{Barlow}) once~$R$ exceeds a random quantity~$R_3(\omega)$, we have
\begin{equation}
\label{E:7.19}
\widetilde V_R^{-1}\sum_{x,y}\widetilde\omega_{xy}\bigl(\varphi(x)\wedge\varphi(y)\bigr)\bigl[\,f(x)-f(y)\bigr]^2
\ge \tilde c_4 R^{-2}\text{\rm Var}_{\kappa}(f)
\end{equation}
for some constant~$\tilde c_4>0$, provided that~$R\ge R_3(\omega)$. (Here is where it is essential that~$\varphi$ is defined using the distance measured on the percolation graph~$\scrC_{\infty,\alpha}$.) The bound $d_\omega(x)\ge(2d)^{-1}\pi_\omega(x)$ for~$x\in\scrC_{\infty,\alpha}$ yields
\begin{multline}
\label{E:7.20}
\qquad\quad
\text{\rm Var}_{\kappa}(f):=\sum_x\kappa(x)\bigl[\,f(x)-E_\kappa(f)\bigr]^2
\\
\ge
\frac{V_R}{\widetilde V_R}\,\frac1{2d}\sum_x\nu(x)\bigl[\,f(x)-E_\kappa(f)\bigr]^2
\ge \frac{V_R}{\widetilde V_R}\,\frac1{2d}\text{\rm Var}_{\nu}(f),
\qquad\quad
\end{multline}
where we noted that the second sum is further decreased when $E_\kappa(f)$ is replaced by~$E_\nu(f)$. (Namely, $a\mapsto \E((Z-a)^2)$ is minimized by~$a=\E Z$.) Since $\widetilde\omega_{xy}\le\alpha^{-1}\hat\omega_{xy}$, \twoeqref{E:7.19}{E:7.20} now yield the claim with~$c_{10}:=\alpha(2d)^{-1}\tilde c_4$.
\end{proofsect}

The core part of the calculation is now finished by noting the following fact:

\begin{lemma}
\label{L:7.7}
Let~$\tilde c=\tilde c(T,R,z)$ be defined by $\tilde c(T,R,z):=\sup_{t\ge T}\sup_y{\cmss q}_t(z,y)V_R$. Then
\begin{equation}
\label{E:7.26}
\text{\rm Var}_{\nu}(w_{z,t})\ge\frac{[\log\tilde c-H_z(t)]^2}{9\tilde c}\Bigl(P_\omega^z\bigl(\text{\rm dist}_\omega(z,\widetilde X_t)\le\tfrac23R\bigr)-9\texte^{2+H_z(t)}\Bigr)
\end{equation}
holds for all~$t\ge T$ and all~$z\in\scrC_{\infty,\alpha}$.
\end{lemma}

\begin{proofsect}{Proof}
This is justified by following the calculation in displays (5.8-5.9) of~\cite{Barlow} just stopping short of substituting the explicit bound (5.2) at the very last step.
\end{proofsect}

Now we are ready to start constructing the proof of the lower bound on~${\cmss q}_t(0,x)$. Suppose $d\ge2$. First we notice that we do not need to prove the desired claim for all $t\ge R^2$ and $|x|\le R$; it suffices to prove it for $t$ a constant multiple larger and~$|x|$ a constant multiple smaller than is dictated by these bounds. (We will find it is easier to prove this using $\ell_\infty$-distances; hence the formulation using those).

\begin{lemma}
\label{lemma-boost}
Let $\alpha\in(0,\alpha_0]$. There exists a $\PP_\alpha$-a.s.\ finite random variable $R_6=R_6(\omega)$ and a constant $c=c(d,\alpha)\in(0,1)$ such that the following is true: If for some decreasing function $s\mapsto\beta(s)\in(0,1)$, a constant $\eta\in(0,\ffrac12)$ and all integers $R\ge R_6$, 
\begin{equation}
\label{E:7:29a}
\min_{\begin{subarray}{c}
x\in\scrC_{\infty,\alpha}\\|x|_\infty\le R
\end{subarray}}\,\,
\min_{\begin{subarray}{c}
y\in\scrC_{\infty,\alpha}\\|y-x|_\infty\le \eta R
\end{subarray}}
{\cmss q}_t(x,y)\ge \beta(tR^{-2})\,\,R^{-d},\qquad t\ge 
R^2/\eta,
\end{equation}
then for all $R\ge 4R_6/\eta$,
\begin{equation}
\label{E:7:30a}
\min_{\begin{subarray}{c}
x\in\scrC_{\infty,\alpha}\\|x|_\infty\le \eta R
\end{subarray}}\,\,
\min_{\begin{subarray}{c}
y\in\scrC_{\infty,\alpha}\\|y-x|_\infty\le R
\end{subarray}}
{\cmss q}_t(x,y)\ge \bigl[c\eta^d\beta(tR^{-2})\bigr]^{4/\eta}\,\,R^{-d},\qquad t\ge 
R^2.
\end{equation}
\end{lemma}

It is worth noting that this does not follow by a simple rescaling of~$R$. Indeed, to reduce the lower bound on the range of $t$ one has to also reduce the separation between~$x$ and~$y$. 

\begin{proofsect}{Proof of Lemma~\ref{lemma-boost}}
Abbreviate $\Lambda_N(x):=x+[-N/2,N/2]^d\cap\Z^d$ and set $p:=\PP(0\in\scrC_{\infty,\alpha})$. By the Spatial Ergodic Theorem, there exists a random variable $R_6'=R_6'(\omega)$ such that, for all $N\ge R_6'$, every box $\Lambda_N(x)$ with $x\in(N\Z)^d$ and $|x|_\infty\le 4N/\eta$, will contain at least $\frac12pN^d$ vertices of~$\scrC_{\infty,\alpha}$. Set $R_6:=4R_6'/\eta$, pick $R\ge R_6$ and note that $N:=\lfloor\eta R/3\rfloor\ge R_6'$. Now pick $x,y\in\scrC_{\infty,\alpha}$ with $|x|_\infty\le\eta R$ and $|y-x|_\infty\le R$ and let $z_0,\dots,z_{m}\in(N\Z)^d$ be a path such that
\begin{equation}
|z_i|_\infty\le R\quad\text{and}\quad|z_{i+1}-z_i|_\infty=N,\qquad i=1,\dots,m-1,
\end{equation}
and
\begin{equation}
|x-z_0|_\infty\le N\quad\text{and}\quad|y-z_{m+1}|_\infty\le N.
\end{equation}
It is not hard to check that such a path exists for $1/\eta\le m<4/\eta-1$. By Chapman-Kolmogorov and the fact that $\pi_\omega(\cdot)\ge\alpha$ on~$\scrC_{\infty,\alpha}$,
\begin{equation}
\cmss q_{mt}(x,y)\ge\alpha^m\sum_{x_1\in\Lambda_N(z_1)\cap\scrC_{\infty,\alpha}}\!\!\dots\!\!\sum_{x_m\in\Lambda_N(z_m)\cap\scrC_{\infty,\alpha}}\,\,\,\prod_{i=0}^{m-1}\cmss q_t(x_i,x_{i+1})
\end{equation}
where $x_0:=x$ and $x_{m+1}:=y$. Since
\begin{equation}
|x_i-x_{i+1}|_\infty\le|x_i-z_i|_\infty+|z_i-z_{i+1}|_\infty+|x_{i+1}-z_{i+1}|_\infty\le3N\le\eta R,
\end{equation}
we are permitted to apply the lower bound \eqref{E:7:29a} to each term in the product. Along with the bound $|\Lambda_N(z_m)\cap\scrC_{\infty,\alpha}|\ge \frac12 pN^d$, this yields
\begin{equation}
\cmss q_{mt}(x,y)\ge\bigl(\tfrac12\alpha p N^d\bigr)^m\bigl(\beta(tR^{-2}) R^{-d}\bigr)^{m+1},\qquad t\ge R^2/\eta.
\end{equation}
Writing $t$ for~$mt$, invoking the monotonicity of $t\mapsto\beta(tR^{-2})$ and the bounds $1/\eta\le m+1<4/\eta$ the claim follows with $c:=\frac12\alpha p 3^{-d}$.
\end{proofsect}

Our task is thus to establish the premise \eqref{E:7:29a} of the previous lemma. We begin by recalling the following bounds from~\cite{BP}: There is a $\PP_\alpha$-a.s.\ finite random variable~$t_0=t_0(\omega)$ and constants $\tilde c_5,\tilde c_6<\infty$ such that for $\PP_\alpha$-a.e.~$\omega$ and all~$t\ge t_0(\omega)$,
\begin{equation}
\label{E:q-bd}
\sup_{z\in K_t}\,\sup_{x\in\scrC_{\infty,\alpha}}{\cmss q}_t(z,x)\le\tilde c_5\,t^{-d/2}
\end{equation}
and
\begin{equation}
\label{E:E-bd}
\sup_{z\in K_t}E_\omega^z\, \dist_\omega(z,\widetilde X_t)\le \tilde c_6\sqrt t.
\end{equation}
These are implied by \cite[Propositions~6.1,6.2]{BP} via the argument (6.33-6.37) in~\cite{BP} and also the fact that the graph-theoretical distance, the Euclidean distance and also the distance associated with the Markov chain~$\hat X$ on~$\scrC_{\infty,\alpha}$ are commensurate; cf~Proposition~\ref{P:perc}(3-5). 
Incidentally, the latter also yields
\begin{equation}
\label{E:V-bd}
V_R\le\tilde c_7 R^d,\qquad R\ge R_4(\omega),
\end{equation}
for some constant~$\tilde c_7<\infty$ and a $\PP_\alpha$-a.s.\ finite random variable  $R_4=R_4(\omega)$.

Introduce the constants $\tilde c:=\tilde c_5\tilde c_7(24\tilde c_6)^d$ and $\tilde c':=\frac12c_{10}(144\tilde c)^{-1}$ and define
\begin{equation}
\gamma:=(2|\log\tilde c|)\vee(2+\log36)\vee\,\frac{(24\tilde c_6)^2}{\tilde c'}\vee\sqrt{\frac{c_8+c_9}{4\tilde c'}}.
\end{equation}
Then we have:

\begin{lemma}
\label{L-gamma}
There is a $\PP_\alpha$-a.s.\ finite random variable~$R_7=R_7(\omega)$ such that for all~$R\ge R_7(\omega)$ and $T:=R^2/(24\tilde c_6)^2$, 
\begin{equation}
\label{E:supremum}
\sup_{T\le t\le2T}H_z(t)> -\gamma,\qquad z\in K_{R/2}.
\end{equation}
\end{lemma}

\begin{proofsect}{Proof}
Setting $R_5(\omega):=\sup\{R\ge0\colon T\le t_0(\omega)\}$, define $R_7(\omega):=\max_{i=1,\dots,5}R_i(\omega)$.
First we note that, for~$t\in[t_0(\omega),2T]$ and~$z\in K_{R/2}$, \eqref{E:E-bd} implies
\begin{equation}
P_\omega^z(\text{\rm dist}_\omega(z,\widetilde X_t)>\tfrac23R\bigr)\le
P_\omega^z(\text{\rm dist}_\omega(z,\widetilde X_t)>\tfrac16R\bigr)\le\tilde c_6\frac{6\sqrt t}R\le\frac12.
\end{equation}
By \eqref{E:q-bd}, Lemma~\ref{L:7.7} holds for our choices of~$T$ and~$\tilde c$, and \eqref{E:7.26} then becomes
\begin{equation}
\label{E:7.29}
\text{\rm Var}_{\nu}(w_{z,t})\ge\frac{[\log\tilde c-H_z(t)]^2}{\tilde c}\Bigl(\frac1{18}-\texte^{2+H_z(t)}\Bigr),
\end{equation}
provided that $R\ge R_4\vee R_5$, $z\in K_{R/2}$ and $t\in[T,2T]$. 

Suppose now that the supremum in \eqref{E:supremum} is less than $-\gamma$ for some~$z\in K_{R/2}$. Then (by $\gamma\le 2+\log 36$) we would have $\texte^{2+H_z(t)}\le\ffrac1{36}$ and, by way of the fact that $(a-h)^2\ge\frac14h^2$ holds whenever $h\le-2|a|$, also $[\log\tilde c-H_z(t)]^2\ge \frac14 H_z(t)^2$. The right-hand side of \eqref{E:7.29} would then be at least $(144\tilde c)^{-1}H_z(t)^2$ for all $t\in[T,2T]$. If~$R\ge R_7$, Lemmas~\ref{L:7.2}, \ref{L:7.3}, \ref{L:7.4}, \ref{L:7.6} then give
\begin{equation}
H_z'(t)\ge 2\tilde c'R^{-2} H_z(t)^2-\frac14(c_8+c_9)R^{-2}\ge \tilde c'R^{-2} H_z(t)^2,\qquad t\in[T,2T].
\end{equation}
But integrating over $t\in[T,2T]$ (and using that $H_z(t)^{-2}$ stays bounded throughout by the assumption that the supremum in \eqref{E:supremum} is less than $-\gamma$) yields 
\begin{equation}
H_z(T)^{-1}\ge \tilde c'TR^{-2}+H_z(2T)^{-1}
\ge \tilde c'TR^{-2}-\frac1\gamma\ge0,
\end{equation}
thus contradicting the assumption that $H(t)\le-\gamma<0$ for all $t\in[T,2T]$.
\end{proofsect}

\begin{proofsect}{Proof of Proposition~\ref{prop-lower}}
As in the continuous setting, a key part of the proof is to show a (linear) lower bound on~$H_z(t)$.  Let $R_0':=R_1\vee R_2\vee R_4\vee R_7$, abbreviate $\zeta:=\frac14(c_8+c_9)$ and recall our notation for~$T$ above. By Lemma~\ref{L-gamma}, there is a~$t'\in[T,2T]$ for which $H_z(t')\ge-\gamma$. The monotonicity of $t\mapsto H_z(t)+\zeta R^{-2}t$ (cf Corollary~\ref{cor-monotone}) shows
\begin{equation}
H_z(t)\ge H_z(t')-\zeta R^{-2}(t-t'),\qquad t\ge t',
\end{equation}
and so
\begin{equation}
\label{E:H-bd}
H_z(t)\ge -\gamma-\zeta t R^{-2}
\end{equation}
holds for all $R\ge R_0'(\omega)$, all $t\ge2T$ and all $z\in K_{R/2}$.

To see how this implies the desired claim, we invoke the Markov property, reversibility and the fact that $\varphi(x)\le1$ to get
\begin{equation}
V_R{\cmss q}_{2t}(x,y)\ge\sum_{z}V_R{\cmss q}_t(x,z)\,V_R\,{\cmss q}_t(y,z)\,\nu(y).
\end{equation}
Taking logs and applying Jensen's inequality, this becomes
\begin{equation}
\log\bigl(V_R{\cmss q}_{2t}(x,y)\bigr)\ge H_x(t)+H_y(t)\ge -2\gamma-2\zeta R^{-2}t,\qquad x,y\in K_{R/2},
\end{equation}
whenever $t\ge 2T$ and $R\ge R_0'$. By the domination $\dist_\omega(x,y)\le d|x-y|_\infty$, the ball $K_{R/2}$ contains all vertices~$x\in\scrC_{\infty,\alpha}$ with $|x|_\infty\le R/(2d)$. A simple rescaling of~$R$ then yields \eqref{E:7:29a} with $\beta(s):=\tilde c_7^{-1}(2d)^{-d}\texte^{-2\gamma-\zeta s/(2d)^2}$, $\eta:=36\tilde c_6^2/d^2$ and $R_6:=R_0'/(2d)$. Invoking Lemma~\ref{lemma-boost} (if $\eta<1$) and the comparison $|x|\le|x|_\infty$, the claim follows.
\end{proofsect}

\section*{Acknowledgments}
\noindent
O.B.\ would like to thank Pierre Mathieu for his support and to dedicate this work to his
father, Youcef Bey. The research of M.B.\ was partially supported by the NSF grant DMS-0949250 and the GA\v CR project P201-11-1558. We wish to thank an anonymous referee for spotting a minor, albeit important, typo in a first write-up of this text.

\end{document}